\documentclass[reqno]{amsart}

\usepackage{caption}
\usepackage{subcaption}
\usepackage{booktabs}
\usepackage{siunitx}
\usepackage[utf8]{inputenc}
\usepackage{hyperref}
\usepackage{graphicx} % Allows including images
\usepackage{booktabs} % Allows the use of \toprule, \midrule and \bottomrule in tables
\usepackage{bbm}
\usepackage{soul}
\usepackage{tikz}
\usepackage{mathrsfs, amsmath, amsthm, tikz-cd, color}
\usetikzlibrary{shadows}
\usepackage{scalerel,stackengine}
\usepackage{cleveref}
 \usepackage{quiver}
\usepackage{float}
\usepackage{array}

\numberwithin{equation}{section}

\newcommand{\vertiii}[1]{{\left\vert\kern-0.25ex\left\vert\kern-0.25ex\left\vert #1 
    \right\vert\kern-0.25ex\right\vert\kern-0.25ex\right\vert}}
\stackMath
\newcommand\reallywidehat[1]{%
\savestack{\tmpbox}{\stretchto{%
  \scaleto{%
    \scalerel*[\widthof{\ensuremath{#1}}]{\kern-.6pt\bigwedge\kern-.6pt}%
    {\rule[-\textheight/2]{1ex}{\textheight}}%WIDTH-LIMITED BIG WEDGE
  }{\textheight}% 
}{0.5ex}}%
\stackon[1pt]{#1}{\tmpbox}%
}
\usepackage{csquotes}       % Quotation marks
\usepackage{microtype}      % Improved typography
\usepackage{amssymb}        % Mathematical symbols
\usepackage{mathtools}      % Mathematical symbols
\usepackage{xcolor}
\usepackage{eucal}
\usepackage[foot]{amsaddr}

\newtheorem{theorem}{Theorem}[section]
\newtheorem{proposition}[theorem]{Proposition}

\newtheorem*{lemma*}{Lemma}

\theoremstyle{definition}

\newtheorem*{definition*}{Definition}
\newtheorem*{proposition*}{Proposition}
\newtheorem{example}{Example}
\newtheorem{remark}[theorem]{Remark}

\newcommand{\R}{\mathbb{R}}

\newcommand{\Ac}{\mathcal{A}}
\newcommand{\Bc}{\mathcal{B}}
\newcommand{\Cc}{\mathcal{C}}
\newcommand{\Ec}{\mathcal{E}}

\newcommand{\Ic}{\mathcal{I}}
\newcommand{\Uc}{\mathcal{U}}

\newcommand{\Sc}{\mathcal{S}}

\newcommand{\Hf}{\mathfrak{H}}
\newcommand{\Bf}{\mathfrak{B}}

\newcommand{\Zf}{\mathfrak{Z}}
\newcommand{\dom}{\mathrm{dom}}

\newcommand{\tr}{\mathrm{tr}}

\newsavebox{\boxedtikzcdbox}

\title{Approximation properties of double complexes}

\author{Daniel F. Holmen\textsuperscript{1,*}}
\author{Jan M. Nordbotten\textsuperscript{1,2}}
\author{Jon E. Vatne\textsuperscript{3}}
\address{*Corresponding author: daniel.holmen@uib.no}
\address{\textsuperscript{1} Center for Modeling of Coupled Subsurface Dynamics, Department of Mathematics, University of Bergen, Allégaten 41, Bergen, Norway}
\address{\textsuperscript{2} Division of Energy and Technology, NORCE Norwegian Research Center AS, Nyg\r{a}rdsgaten 112, 5008 Bergen Norway}
\address{\textsuperscript{3} Department of Economics, BI Norwegian Business School, Kong Christian Frederiks plass 5, Bergen, Norway}

\date{}
\begin{document}
\begin{abstract}
We consider the simplicial de Rham complex and the \v{C}ech-de Rham complex, two bigraded Hilbert complexes whose Hodge-Laplace problems govern spatially coupled problems in mixed dimension and homogeneous dimension, respectively. The former complex can be realized as a subcomplex of the latter.
%It has previously been shown that the simplicial de Rham complex can be embedded as a subcomplex of the \v{C}ech-de Rham complex \cite{holmen2024injective}. 
In this paper, we quantify how close these complexes are to each other by constructing bounded cochain complexes between them, and thus we quantify how close a mixed-dimensional formulation of a problem is to an equidimensionally coupled formulation of the same problem. 

From this construction, we derive a priori- and a posteriori error estimates between the associated Hodge-Laplace problems on the two complexes. These estimates represent the error which is introduced by treating a spatially coupled problem as mixed-dimensional, rather than an equidimensional problem with thin overlaps. 

%We answer this question by relying on the theory from finite element exterior calculus, and constructing a bounded cochain projection from the domain complex of the Hilbert \v{C}ech-de Rham complex to the embedded simplicial de Rham complex. From this construction, 
\end{abstract}

\maketitle

%\tableofcontents

\section{Introduction} \label{section: intro}
Hilbert complexes are obtained by combining the familiar Hilbert spaces from functional analysis with cochain complexes from homological algebra, providing a rich mathematical structure for modeling and analysis of elliptical PDEs. The prototypical example of such a cochain complex is the de Rham complex, and the associated Hodge-Laplace problem on the de Rham complex confines both scalar and vector-Laplace problems with corresponding boundary conditions.

In the framework of finite element exterior calculus (FEEC) \cite{arnoldFEEC3, arnoldFEEC2}, a bounded cochain projection is constructed from infinite-dimensional Hilbert complexes to their finite-dimensional counterparts. These projections lead to several important results related to the consistency, stability and convergence of mixed finite element methods. 

Using the same concepts and similar techniques, we construct a bounded cochain projection between two infinite-dimensional Hilbert complexes; the \v{C}ech-de Rham complex \cite{bott-tu, cdrhodge} and the embedded mixed-dimensional de Rham complex \cite{mdG, holmen2024injective}. 

Concretely, we obtain approximation properties between the Hodge-Laplace problems on the simplicial de Rham complex and the \v{C}ech-de Rham complex, which have been shown to correspond to two different modeling paradigms for spatially coupled problems. Such model approximations for elliptic equations appear in e.g. \cite{koppl2018mathematical, list2020rigorous}.

This article is structured as follows: \cref{section: intro} gives a overview over the theory of Hilbert complexes, as well as introducing the specific Hilbert complexes of interest: the simplicial de Rham complex and the \v{C}ech-de Rham complex. \Cref{section: bcp} discusses Hilbert subcomplexes and bounded cochain projections. In \cref{section: error estimates}, we derive a priori error estimates based on the abstract Hilbert complex theory from \cite{arnoldFEEC3}. We also derive functional a posteriori error estimates based on \cite{pauly2020solution}. Finally, in \cref{sec:examples} we consider a simple coupled problem as an example: two one-dimensional elastically joined rods.

\subsection{Hilbert complexes}\label{subsection: HC}
We begin by establishing basic definitions and notations for Hilbert complexes. We refer to \cite{arnoldFEEC3, arnoldFEEC2} for a more comprehensive study on the topic. 

A \emph{Hilbert complex} $(A^\bullet, d_A)$ is a cochain complex where each object is a Hilbert space and each differential $d_A^k: A^k \to A^{k+1}$ is a closed densely-defined unbounded linear operator:
\begin{equation}
 % https://q.uiver.app/#q=WzAsNSxbMSwwLCJBXntrLTF9Il0sWzAsMCwiLi4uIl0sWzIsMCwiQV57a30iXSxbMywwLCJBXntrKzF9Il0sWzQsMCwiLi4uIl0sWzEsMF0sWzAsMiwiZF57ay0xfSJdLFsyLDMsImRee2t9Il0sWzMsNF1d
\begin{tikzcd}
	{...} & {A^{k-1}} & {A^{k}} & {A^{k+1}} & {...}
	\arrow[from=1-1, to=1-2]
	\arrow["{d^{k-1}}", from=1-2, to=1-3]
	\arrow["{d^{k}}", from=1-3, to=1-4]
	\arrow[from=1-4, to=1-5]
\end{tikzcd}   
\end{equation}
We frequently omit the subscript indicating the domain of the differential operator whenever convenient. Moreover, all Hilbert complexes we are working with are closed (i.e. a Hilbert complex where $\mathrm{im} \; d_A^k$ is closed in $A^{k+1}$ for every integer $k$), and we therefore assume throughout this section that  $(A^\bullet, d_A)$ is a closed Hilbert complex.

Since the differentials $d_A$ are not necessarily defined everywhere on $A^k$, there is an associated subcomplex of $A^k$ which we refer to as the \emph{domain complex}, defined by the domains of the differential operators $d_A$ (we write $\dom\; d^k_A = \Ac^k $):
\begin{equation}
% % https://q.uiver.app/#q=WzAsNSxbMSwwLCJcXG1hdGhjYWx7QX1ee2stMX0iXSxbMiwwLCJcXG1hdGhjYWx7QX1ee2t9Il0sWzMsMCwiXFxtYXRoY2Fse0F9XntrKzF9Il0sWzAsMCwiLi4uIl0sWzQsMCwiLi4uIl0sWzAsMSwiZF57ay0xfSJdLFsxLDIsImReayJdLFszLDBdLFsyLDRdXQ==
\begin{tikzcd}
	{...} & {\mathcal{A}^{k-1}} & {\mathcal{A}^{k}} & {\mathcal{A}^{k+1}} & {...}
	\arrow[from=1-1, to=1-2]
	\arrow["{d^{k-1}}", from=1-2, to=1-3]
	\arrow["{d^k}", from=1-3, to=1-4]
	\arrow[from=1-4, to=1-5]
\end{tikzcd}
\end{equation}
%The domain complex is equipped with a graph inner product 
%\begin{equation}
%\langle u,v \rangle_{\Ac^k} = \langle u,v \rangle + \langle du, dv \rangle_{A^{k+1}},
%\end{equation}
%which induces a graph norm 
%\begin{equation}
%\|u\|_{\Ac^k}^2 = \|u\|_{A^k}^2 + \|du\|^2_{A^{k+1}}.
%\end{equation}

Throughout the text we will employ the convention of using capital letters (e.g. $A, B, C$) for Hilbert complexes and calligraphic letters (e.g. $\Ac, \Bc, \mathcal{C}$) for their respective domain complexes. For the subspaces of $A^k$, we use Fraktur letters to denote the image $\mathfrak{B}^k = \mathrm{im} \; d^{k-1} \subset A^k$ and the kernel $\mathfrak{Z}^k = \ker d^k \subset A^k$. Since $d^2 = 0$, we have that $\mathfrak{B}^k \subset \mathfrak{Z}^k$, and we write $\mathfrak{H}^k$ for the cohomology $\mathfrak{Z}^k/\mathfrak{B}^k$. When it is not immediately clear which Hilbert complex the subspaces belong to, we will write $\mathfrak{B}^k(A)$, $\mathfrak{Z}^k(A)$ and $\mathfrak{H}^k(A)$ for emphasis. 

Since each $A^k$ is a Hilbert space, there is an associated inner product $\langle u,v \rangle_{A^k}$ for $u,v \in A^k$ which defines a norm $\|v\|_{A^k} = \sqrt{\langle v, v \rangle_{A^k}}$. In addition to the inner product and the norm on the full Hilbert complex, we also have a \emph{graph inner product} and a \emph{graph norm} associated to the domain complex $(\Ac^\bullet, d)$:
\begin{equation} \label{eq: graph ip + graph norm}
\langle u,v \rangle_{\Ac^k} = \langle u,v \rangle_{A^k} + \langle du, dv \rangle_{A^{k+1}}, \qquad \|v\|_{\Ac}^2 = \|v\|_{A^k}^2 + \|d v\|_{A^{k+1}}^2.
\end{equation}

Since we have an inner product, we can define the adjoint of the differential operator $d$:
\begin{equation}
\langle du, v \rangle_{A^{k+1}} = \langle u, d^*v \rangle_{A^k}, \qquad u \in A^k, v \in A^{k+1}.
\end{equation}
The adjoint operator $d^*$ forms a chain complex called the \emph{adjoint complex} which also has an associated domain complex:
\begin{equation}
% https://q.uiver.app/#q=WzAsNSxbMSwwLCJcXG1hdGhjYWx7QX1eKl97ay0xfSJdLFsyLDAsIlxcbWF0aGNhbHtBfV97a31eKiJdLFszLDAsIlxcbWF0aGNhbHtBfV4qX3trKzF9Il0sWzAsMCwiLi4uIl0sWzQsMCwiLi4uIl0sWzIsMSwiZF4qX3trKzF9IiwyXSxbMCwzXSxbNCwyXSxbMSwwLCJkXipfayIsMl1d
\begin{tikzcd}
	{...} & {\mathcal{A}^*_{k-1}} & {\mathcal{A}_{k}^*} & {\mathcal{A}^*_{k+1}} & {...}
	\arrow[from=1-2, to=1-1]
	\arrow["{d^*_k}"', from=1-3, to=1-2]
	\arrow["{d^*_{k+1}}"', from=1-4, to=1-3]
	\arrow[from=1-5, to=1-4]
\end{tikzcd}
\end{equation}
We denote subspaces of the adjoint complexes by $\mathfrak{B}^*_k = \mathrm{im} \; d^*_{k+1}$ and $\mathfrak{Z}_k^* = \ker d^*_k$. By slight abuse of notation, we use the same symbol for the space of harmonic forms $\Hf^k =\Zf^k \cap \Zf^*_k$, as it is isomorphic to the cohomology space $\Zf^k/\Bf^k$.

In general, we work with an inner product that can be weighted by a bijective, symmetric, bounded linear operator $w_k$: 
\begin{equation}
\langle u, v \rangle_{A^k_w} = \langle w_k u, v \rangle_{A^k}.
\end{equation}
The weighted inner product can be used to account for physical properties such as material law, or a length scale $\epsilon$ for mixed-dimensional representations (where $\epsilon$ is the physical length in the collapsed dimension).

Note that we in principle have different weights at different degrees of our cochain complex. The differential operator for a weighted Hilbert complex remains the same as in the case of unit weights. The adjoint on the other hand is defined through the inner product and will therefore depend on the choice of weights $w_k$.

We emphasize the weight in the adjoint by writing $d^{*, w}$. The weighted adjoint can be expressed in terms of the unit weighted adjoint through the following relationship:
\begin{equation}
d^{*, w}_k = w^{-1}_{k-1} d^{*,1}_k w_k,
\end{equation}
where $d^{*,1}_k$ is the adjoint with respect to the unweighted inner product.
%We require the weights to be symmetric, positive definite operators $w_k: A^k \to A^k$. 
The differential and the codifferential together define a generalization of the usual Laplace operator in this abstract Hilbert complex setting, called the \emph{Hodge-Laplace operator} or \emph{Hodge-Laplacian}:
\begin{equation}
\Delta_{d,w}^k = d^{k-1} d^{*,w}_k + d_{k+1}^{*,w} d^k.
\end{equation}

The Hodge-Laplace operator in turn defines the problem of interest, the \emph{Hodge-Laplace problem}: given $f \perp \ker(\Delta_{d,w}^k)$ find $v \in \dom(\Delta_{d,w}^k)$ such that
\begin{equation}\label{eq: Hodge-Laplace problem}
\Delta_{d,w}^k v = f, \qquad v \perp \Hf^k.
\end{equation}

The following orthogonal decomposition of $A^k$ is valid for each integer $k$:
\begin{align}\label{eq: hodge decomposition}
A^k &= \mathrm{im \;} d^{k-1} \oplus \ker \Delta_{d,w}^k \oplus \mathrm{im \;} d^{*,w}_{k+1} = \mathfrak{B}^k \oplus \mathfrak{H}^k \oplus \mathfrak{B}^{*,w}_k.
\end{align}
The sub/superscript indicating that the adjoint differential and the Hodge-Laplacian are weighted will regularly be suppressed. A similar decomposition can be achieved for non-closed Hilbert complexes by taking the closure of $\Bf^k$ and $\Bf^{*,w}_k$. The Hodge decomposition allows for defining projection operators on $A^k$, for each of the three orthogonal subspaces. We write $P_\Bf^k$, $P_\Hf^k$ and $P_{\Bf^*}^k$ for projections from $A^k$ onto, $\Bf^k$, $\Hf^k$ and $\Bf^{*,w}_k$, respectively.

By replacing the right-hand side of \cref{eq: Hodge-Laplace problem} with $f - P_\Hf f$, we can formulate the Hodge-Laplace problem for initial data $f \in A^k$, without explicitly requiring orthogonality to $\ker(\Delta_{d,w}^k)$.

Finally, we have a Poincaré inequality for each integer $k$, which tells us that there exists a constant $C_{\Ac, k}$ (called the Poincaré constant) such that
\begin{equation}\label{eq: poincare ineq}
\|v\|_{\Ac^k} \leq C_{\Ac, k} \|dv\|, \qquad \forall  v \in \Ac^k \cap \Zf^{k, \perp},
\end{equation}
where $\Zf^{k, \perp}$ is the orthogonal complement of $\Zf^k$ in $A^k$. From the Hodge decomposition \cref{eq: hodge decomposition}, we can identify $\Zf^{k,\perp}$ with $\mathfrak{B}^{*,w}_k$. The inequality \cref{eq: poincare ineq} relies on $(A^\bullet, d)$ being a closed Hilbert complex.

\subsection{Double Hilbert complexes}
Ultimately, we are interested in Hilbert complexes defined on a hierarchy of domains rather than a single domain. We therefore consider products of Hilbert spaces for each domain in the hierarchy, which leads to Hilbert complexes which are bigraded.

A \emph{double cochain complex} is a lattice of bigraded objects $(A^{p,q}, d_h, d_v)$ with a horizontal and vertical differential, 
\begin{subequations}    
\begin{align}
d_h: & \Ac^{p,q} \subset A^{p,q} \to A^{p+1,q}, \\
d_v: &\Ac^{p,q} \subset A^{p,q} \to A^{p,q+1},
\end{align}
\end{subequations}

where each row $A^{\bullet, q}$ is a cochain complex, each column $A^{p, \bullet}$ is a cochain complex and the two differentials anti-commute:
\begin{subequations}
\begin{align}
d_h d_h &= 0, \\
d_v d_v &= 0, \\
d_v d_h &= - d_h d_v.
\end{align}
\end{subequations}

We assume the double complex to be in the first quadrant, i.e. both $p \geq 0$ and $q \geq 0$. Moreover, we will work with double complexes that are bounded with respect to both the degree $p$ and $q$, meaning that $A^{p,q} = 0$ if either $p>p_{max}$ or $q>q_{max}$.

Given a double complex $A^{p,q}$, we define the \emph{total complex} of $A^{p,q}$ as the single-graded cochain complex 
\begin{equation}
(\mathrm{Tot}A^{p,q})^k = A^k = \bigoplus_{k = p + q} A^{p,q}, \qquad D = d_h + d_v.
\end{equation}
Since we assume that the bigraded double complex is bounded with respect to both degree $p$ and $q$, the total complex is also necessarily bounded.

In our case, we are interested in double complexes where each $A^{p,q}$ is a Hilbert spaces, and both the horizontal differential $d_h$ and vertical differential $d_v$ are closed densely-defined unbounded linear operators. Moreover, we require $d_v d_h$ and $d_h d_v$ to be well-defined. Hence, we require that $\mathrm{im}(d_v^{p,q}) \subset \dom(d_h^{p,q+1})$ and  $\mathrm{im}(d_h^{p,q}) \subset \dom(d_v^{p+1,q})$.

Each of the Hilbert spaces $A^{p,q}$ has an inner product, which allows us to define an inner product on the total complex $A^k$:
\begin{equation}
\langle u, v \rangle_{A^k} = \sum_{p+q=k}  \langle u_{p,q}, v_{p,q} \rangle_{A^{p,q}}, \qquad u_{p,q}, v_{p,q} \in A^{p,q}.
\label{eq: double inner product}
\end{equation}

In general, the domain of $D = d_h + d_v$ is the intersection $\dom (d_h) \cap \dom(d_v)$, which may not be dense. For two unbounded operators $d_1$ and $d_2$, a sufficient condition for $\dom(d_1 + d_2)$ to be dense is given in \cite{lennon1974sums}. Similarly, a sufficient condition for $d_1 + d_2$ to be closed is presented in \cite[Chapter 4, Theorem 1.1]{kato2013perturbation}. In practice, these conditions are satisfied for the Hilbert complexes of interest in this paper.

Under the assumption that $\dom (d_h) \cap \dom(d_v)$ is dense and $D = d_h + d_v$ is closed, the total complex $(A^\bullet, D)$ is a Hilbert complex. The domain complex $\Ac^k \subset A^k$ is equipped with the graph norm $\|v\|^2_{\Ac^k} = \|v\|^2_{A^k} + \|Dv\|^2_{A^{k+1}}$.

\subsection{\v{C}ech-de Rham complex and simplicial de Rham complex}
The Hilbert complex of interest to us is the de Rham complex with square-integrable function components. The exterior derivative is not defined everywhere on this complex, but it is defined on a dense subcomplex. We refer to this complex as the $L^2$ de Rham complex, and write $(L^2\Lambda^\bullet, d)$. The domain of the exterior derivative defines a subcomplex which we refer to as the \emph{domain complex}, for which we write $(H\Lambda^\bullet, d)$. For more details on the de Rham complex, see \cite[Chapter 6]{arnoldFEEC}.

We consider a hierarchy of manifolds that are overlapping, each equipped with a de Rham complex. In the case of the \v{C}ech-de Rham complex we have overlapping manifolds, and in the case of the mixed-dimensional de Rham complex each lower-dimensional manifold of dimension $m$ is contained in the boundary of some manifold with dimension $m+1$. By defining a difference operator and a jump operator on these hierarchies we get two double complexes.

Let $\Uc = \{U_i\}_{i \in \Ic}$ be a finite open cover of a manifold $\mathbf{\Omega}$ indexed by an ordered set $\Ic$. We write $i \in \Ic^p$ for an increasing multi-index $(i_0, ..., i_p)$ where each $i_j \in \Ic$, and write $U_i$ with $i \in \Ic^p$ as short-hand for the intersection $U_{i_0} \cap ... \cap U_{i_p}$. We form a product space by taking the de Rham complex of each intersection $U_i$, for each increasing multi-index $i \in \Ic^p$:
\begin{equation}\label{eq: bigraded CdR}
C^{p,q} = C^p(\Uc, L^2\Lambda^q) = \prod_{i \in \Ic^p} L^2\Lambda^q(U_i).
\end{equation}
The spaces $C^{p,q}$ make up the bigraded \v{C}ech-de Rham complex. This product admits a differential operator in the exterior derivative by having it act on each de Rham complex. Moreover, we can define a \emph{difference operator} by taking alternating sums restricted to intersections of degree $p+1$:
\begin{align}
(\delta v)_i &=  \sum_{j=0}^{p+1} (-1)^{k+j}  \; v_{i \setminus j}|_{U_i}, &
\forall i \in \mathcal{I}^{p+1},
\end{align}
where $k=p+q$ and $i \setminus j$ is short-hand for $i_0, ..., i_{j-1}, i_{j+1}, ..., i_{p+1}$, i.e. the multi-index $i$ with the $j$-th entry removed. We write $\Cc^{p,q} = \prod_{i \in \Ic^p} H\Lambda^q(U_i)$ for the domain complex. The difference operator and the exterior derivative together  form a double complex:
\begin{align}
\delta^p:& \Cc^{p,q} \to \Cc^{p+1, q} \\
d^q:& \Cc^{p,q} \to \Cc^{p,q+1}
\end{align}
We are able to form a total complex $\Cc^k = \bigoplus_{p+q=k} \Cc^{p,q}$ with total differential $D^k = \bigoplus_{p+q=k} d^q + \delta^p$.

The simplicial de Rham complex shares a lot of similarities to the \v{C}ech-de Rham complex. We consider a partition of $\mathbf{\Omega}$ in a collection of non-overlapping $n$-manifolds $\{\Omega_i\}_{i \in \Ic}$ indexed by the same index set $\Ic$, such that $\bigcup_{i\in\Ic} \overline{\Omega_i}=\overline{\mathbf{\Omega}}$, and such that the common boundary of two $n$-manifolds $\Omega_{i_0}$ and $\Omega_{i_1}$ is a new manifold $\Omega_i$, $i = (i_0, i_1) \in \Ic^1$. More generally, a manifold of dimension $(n-p)$ is then represented by some multi-index $i \in \Ic^p$. Details regarding the geometrical description is found in \cite{holmen2024injective} (see also \cite{mdG} for a more general mixed-dimensional Hilbert complex).

The choices we make in this work ensure that the open cover $\Uc$ for the \v{C}ech-de Rham complex is not just an open cover of a manifold $\mathbf{\Omega}$, it also satisfies that each $U_i \in \Uc$ covers $\Omega_i \in \{\Omega_i\}_{i\in \Ic}$. More generally, any $U_j$ covers $\Omega_j$, for the same multi-index $j \in \Ic^p$. Moreover, in \cref{proposition: liminf} we consider a family of open covers $\{\Uc_\epsilon\}_{\epsilon \in (0,\epsilon_M)}$ indexed by a non-negative parameter $\epsilon$ which denotes the maximal diameter for overlaps $U_i$, where $i \in \Ic^p$, $p \neq 0$.
For some sections, we consider a weighted inner product for the \v{C}ech-de Rham complex, where we use the parameter $\epsilon$ as a weight, with exponent matching the degree of overlap (which corresponds to codimension in the mixed-dimensional setting):
\begin{equation}\label{eq: weighted cdr inner products}
\langle u_{p,q},v_{p,q} \rangle_{C^{p,q}_\epsilon} = \langle \epsilon^{-p}  u_{p,q},v_{p,q} \rangle_{C^{p,q}}, \qquad \langle u, v  \rangle_{C^{k}_\epsilon} = \sum_{p+q=k}\langle u_{p,q},v_{p,q} \rangle_{C^{p,q}_\epsilon}.
\end{equation}

In the same vein as for the \v{C}ech-de Rham complex in \cref{eq: bigraded CdR}, we define a bigraded product space for the mixed-dimensional hierarchy $\{\Omega_i\}_{i \in \Ic}$:
\begin{equation}
S^{p,q} = \prod_{i \in \Ic^p} L^2\Lambda^q(\Omega_i).
\end{equation}

The \emph{jump operator} is defined similarly to the difference operator from the \v{C}ech-de Rham complex, but in this case we take the trace of differential forms onto boundaries rather than restricting it onto a subset.
\begin{align}
(\delta v)_i &=  \sum_{j=0}^{p+1} (-1)^{k+j}  \; \tr \; v_{i \setminus j}, &
\forall i \in \mathcal{I}^{p+1},
\end{align}
Here, the trace is understood as a mapping of differential forms on $\Omega_{i \setminus j}$ to differential forms on $\Omega_i$. For the domain complex, we want the composition of the two differentials to be well-defined. However, the trace of an element in $H\Lambda^q(\Omega)$ is not necessarily in $H\Lambda^q(\partial \Omega)$. To account for this, we introduce recursively defined Sobolev spaces:
\begin{equation}
H\Lambda^q(\Omega_i, \tr) = \{v \in H\Lambda^q(\Omega_i) : \tr_j \; v \in H\Lambda^q(\Omega_j, \tr) \}.\label{eq: recursive trace spaces}
\end{equation}
These spaces ensure that the trace of a differential form is in the domain of the exterior derivative $d$, as well as allowing for iterative application of the trace operator.

We define the simplicial de Rham complex to be the bigraded cochain complex consisting of the Sobolev spaces from \cref{eq: recursive trace spaces}:
\begin{equation}
\Sc^{p,q} = \prod_{i \in \Ic^p} H\Lambda^q(\Omega_i, \tr), \qquad \Sc^k = \bigoplus_{p+q=k} \Sc^{p,q}.
\end{equation}

%~2 paragraphs on inner products, graph inner products/norms, weighted inner products to include material law and dimension.

In previous work, we have constructed a bounded injective cochain map between the two aforementioned double Hilbert complexes \cite{holmen2024injective}. That is, a bigraded cochain map
\begin{equation}
\Xi^{p,q}: \Sc^{p,q} \to \Cc^{p,q},
\end{equation}
which induces a cochain map between the total complexes:
\begin{equation}
\Xi^{k}: \Sc^{k} \to \Cc^{k}. \label{eq: embedding}
\end{equation}
Moreover, we prove that the cochain map is bounded from below and above (see \cite[Section 4.1]{holmen2024injective}): for each pair $(p,q)$ satisfying $0 \leq p+q \leq n$, there exists constants $C_1$ and $C_2$ such that
\begin{equation} \label{eq: bounded cochain map}
C_1 \|v\|_\Sc \leq \|\Xi^{p,q}(v)\|_\Cc \leq C_2 \|v\|_\Sc, \qquad \forall v \in \Sc^{p,q}.
\end{equation}
The boundedness of the cochain map means that we have an isomorphism from $\Sc$ onto its image by $\Xi^k$, and we therefore have a subcomplex $\Ec^\bullet = \Xi^\bullet(\Sc)$. The \v{C}ech-de Rham complex together with this subcomplex becomes an appropriate setting for an approximation theory between equidimensional- and mixed-dimensional models.

\section{Subcomplexes and bounded cochain projections}\label{section: bcp}
In this section, we will look at two of the fundamental properties in the theory of Hilbert complexes; subcomplexes and the existence of bounded cochain projections. In the theory of FEEC, the existence of bounded cochain projections onto Galerkin subspaces implies stability of the associated finite element method. Analogously, bounded cochain projections from the \v{C}ech-de Rham complex $\Cc^\bullet$ to the subcomplex $\Ec^\bullet$ will provide an approximation theory between solutions of equidimensional problems and solutions of mixed-dimensional problems.

We will review the definitions of bounded cochain projections, show sufficient conditions for existence of bounded cochain projections and review the error estimates that are derived from Hilbert subcomplexes.

\subsection{Bounded cochain projections}
Let $(A^\bullet, d)$ be a Hilbert complex, $(\Ac^\bullet, d)$ be its associated domain complex and suppose that $(\Bc^\bullet, d|_{\Bc})$ is a subcomplex of $(\Ac^\bullet, d)$. That is, for each $k$ we have $\Bc^k \subset \Ac^k$ and $d\Bc^k \subset \Bc^{k+1}$.

The differential operator on $\Bc$ is just the restriction $d|_{\Bc}$, so we can simply write $d$ for both complexes. Similarly, the inner product on $\Bc^k$ is then just restriction of the inner product on $\Ac^k$, so we do not distinguish between $\langle \cdot, \cdot \rangle_\Ac$ and $\langle \cdot, \cdot \rangle_{\Bc}$. Note that the adjoint $d^*_\Bc$ on the subcomplex is not generally equal to the restriction $d^*|_\Bc$.

A \emph{bounded cochain projection} is a collection of linear maps $\pi^k: \Ac^k \to \Bc^k$ which restricts to the identity on $\Bc^k$, commutes with the differential operators (meaning that $\pi^{k+1} \circ d = d \circ \pi^k$) and satisfies the following inequality: there exists a positive constant $\kappa_k$ such that
\begin{equation}\label{eq. bcp1}
\| \pi^k v \|_{\Bc^k} \leq \kappa_k \|v\|_{\Ac^k}, \qquad \forall v \in \Ac^k. 
\end{equation}
Here, the inequality can be in either the full norm ($\|\cdot\|_{A^k}$) or the graph norm (\cref{eq: graph ip + graph norm}). 

\begin{proposition}
Boundedness in the full norm $\|\cdot\|_{A^k}$ implies boundedness in the graph norm $\|\cdot\|_{\Ac^k}$.    
\end{proposition}
\begin{proof}
We assume that for every $k$,
\begin{equation}
\|\pi^k v \| \leq \kappa_k \|v\|, \qquad \forall v \in \Ac^k. 
\end{equation}
Using the commutative property of the cochain projection, we have 
\begin{equation}
\|\pi^k v\|^2_{\Ac^k} = \|\pi^k v\|^2 + \|d \pi^k v \|^2 =  \|\pi^k v\|^2 + \| \pi^{k+1} d v \|^2.
\end{equation}
Using the boundedness assumption, we get
\begin{equation}
\|\pi^k v\|^2 + \| \pi^{k+1} d v \|^2\leq \kappa^2_k \|v\|^2 + \kappa^2_{k+1} \|dv\|^2 \leq \max\{\kappa_k^2, \kappa_{k+1}^2\} \|v\|^2_{\Ac^k}.
\end{equation}
Hence we have boundedness in the graph norm.
\end{proof}

A bounded cochain projection gives rise to a Poincaré inequality between the two complexes, meaning that there exists a constant $ c_{\Bc, k}$ such that
\begin{equation}
\|v \|_{\Ac^k} \leq c_{\Bc, k} \|\pi^k\| \|dv \|, \qquad v \in \Bc^k \cap \Zf^{k,\perp},
\end{equation}
where $\|\pi^k\|$ is the operator norm of the bounded cochain projection $\pi^k$.

\subsection{Different formulations of the Hodge-Laplace problem} %The Hodge-Laplace problem on the de Rham complex can be viewed as a generalization of the scalar/vector Laplace problem. Moreover, the Hodge-Laplace problem on other Hilbert complexes can represent other type of problems, for instance the Hodge-Laplace problem on the elasticity complex gives rise to Hooke's law and Cauchy momentum equation \cite{AFWelasticity}.

Herein we briefly summarize the many different formulations of the Hodge-Laplace problem: the primal strong formulation, the primal weak formulation and the mixed weak formulation. In addition, we may also formulate the Hodge-Laplacian as a system of first order equations, which is later used for deriving a posteriori error estimates. We also consider the mixed weak formulation of the Hodge-Laplace problem on the subcomplex $\Bc^\bullet \subset \Ac^\bullet$.

We refer to \cref{eq: Hodge-Laplace problem} as the \emph{primal strong formulation} of the Hodge-Laplace problem. Rather than insisting that $f \perp \ker(\Delta_{d,w}^k)$, we can write \cref{eq: Hodge-Laplace problem} as
\begin{equation}\label{eq: Hodge-Laplace primal strong}
\Delta_{d,w}^k v = f - P_\Hf f, \qquad v \perp \Hf^k.
\end{equation}
where $P_\Hf$ denotes the projection associated with the Hodge decomposition in \cref{eq: hodge decomposition}. 

In certain applications, the initial data $f$ is sometimes entirely in the space $\Bf^k$. In this case, the Hodge-Laplace problem is formulated as follows: given $f \in \Bf^k$, find $v \in \Bf^k$ such that
\begin{equation}\label{eq: B-problem}
D^{k-1}D^*_{k}v=f, \qquad v \perp \Hf^k.
\end{equation} 
Likewise, given $g \in \Bf^*_k$, the Hodge-Laplace problem becomes to find $v \in \Bf^*_k$ such that $D^*_{k+1}D^{k} v = g$. We refer to these two special cases as $\Bf$-problems and $\Bf^*$-problems, respectively. Note that the Hodge-Laplace problems for $k=n$ and $k=0$ are $\Bf$-problems and $\Bf^*-$problems, respectively.

To write the Hodge-Laplacian in a weak formulation, we integrate with respect to a test-differential form $v'$ which vanishes on the boundary of the domain:
\begin{equation}
\langle (d^{k-1} d^{*,w}_k + d_{k+1}^{*,w} d^k) v, v' \rangle = \langle f - P_\Hf f, v' \rangle, \qquad v \perp \Hf^k.
\end{equation}
Since we assume $v'$ to vanish on the boundary, we can use linearity of the inner product and integration by parts to write the equation in the \emph{primal weak formulation}:
\begin{equation}\label{eq: Hodge-Laplace primal weak}
\langle dv, dv' \rangle + \langle d^*v, d^*v' \rangle = \langle f - P_\Hf f, v' \rangle, \qquad v \perp \Hf^k.
\end{equation}
We introduce the following variables: $u = d^*v$ and $q = P_\Hf f$, and obtain the \emph{mixed weak formulation} of the Hodge-Laplace problem: find $(u,v,q) \in \Ac^{k-1} \times \Ac^{k} \times \Hf^k$ such that the following set of equations hold:
\begin{subequations}\label{eq: Hodge-Laplace mixed weak}
\begin{align}
\langle u, u' \rangle - \langle v, du' \rangle &= 0, \qquad &\forall u' \in \Ac^{k-1}, \\
\langle du, v' \rangle + \langle dv, dv' \rangle + \langle q, v' \rangle &= \langle f, v' \rangle, \qquad &\forall v' \in \Ac^k, \\
\langle v, q' \rangle &= 0, \qquad &\forall q' \in \Hf^k(\Ac).
\end{align}
\end{subequations}

The three different formulations of the Hodge-Laplace problem \cref{eq: Hodge-Laplace primal strong}, \cref{eq: Hodge-Laplace primal weak} and \cref{eq: Hodge-Laplace mixed weak} are all equivalent and well-posed \cite[Theorem 4.7 and 4.8]{arnoldFEEC}.

%Include discrete/subcomplex HL problem
In addition to the three main formulations of the Hodge-Laplace problem, we can also formulate a subcomplex mixed formulation of the Hodge-Laplace problem. We write $\Zf^k(\Bc)$ and $\Bf^k(\Bc)$ for the subspace kernel and subcomplex image, respectively. Note that $\Zf^k(\Bc) \subset \Zf^k(\Ac)$ and likewise $\Bf^k(\Bc) \subset \Bf^k(\Ac)$. We define the subspaces of harmonic forms on $\Bc^\bullet$ by
\begin{equation}
\Hf^k(\Bc) = \{v \in \Zf^k(\Bc): v \perp \Bf^k(\Bc) \}.
\end{equation}
In general, $\Hf^k(\Bc) \not\subset \Hf^k(\Ac)$. The mixed formulation of the subcomplex Hodge-Laplace problem is as follows: given $f \in A^k$, find $(u_\Bc, v_\Bc, q_\Bc) \in \Bc^{k-1} \times \Bc^k \times \Hf^k(\Bc)$ such that
\begin{subequations}\label{eq: Hodge-Laplace subcomplex mixed weak}
\begin{align}
\langle u_\Bc, u' \rangle - \langle v_\Bc, du' \rangle &= 0, \qquad &\forall u' \in \Bc^{k-1}, \\
\langle du_\Bc, v' \rangle + \langle dv_\Bc, dv' \rangle + \langle q_\Bc, v' \rangle &= \langle f, v' \rangle, \qquad &\forall v' \in \Bc^k, \\
\langle v_\Bc, q' \rangle &= 0, \qquad &\forall q' \in \Hf^k(\Bc).
\end{align}
\end{subequations}

\subsection{Existence of bounded cochain projections}
%It is known that under certain assumptions, there is guaranteed to exist bounded cochain projections. 
In this section we will prove that there exists bounded cochain projections from the \v{C}ech-de Rham complex to the subcomplexes given by \cref{eq: embedding}.

We have the following two assumptions for the subcomplex $(\Bc, d)$: there exists positive constants $\kappa_1$ and $\kappa_2$ such that 
\begin{subequations}\label{eq: assumpions}
\begin{align}\label{eq: ass1}
\|v\|_\Ac &\leq \kappa_1 \|dv\|_\Ac, \qquad &v \in \mathfrak{Z}^{k,\perp}(\Bc), \\
\|q\|_\Ac &\leq \kappa_2 \|P_{\mathfrak{H}(\Ac)} q\|_\Ac, \qquad &q \in \Hf^k(\Bc). \label{eq: ass2}
\end{align}   
\end{subequations}

Under the assumptions from \cref{eq: ass1} and \cref{eq: ass2}, the existence of bounded cochain projections is guaranteed, as shown in \cite{arnoldFEEC2}.

\begin{theorem}\label{theorem: existence of bcp}
Suppose the assumptions in \cref{eq: assumpions} hold. Then there exists bounded cochain projections $\pi^k: \Ac^k \to \Bc^k$, where $\|\pi^k\|$ is bounded by the constants $\kappa_1$ and $\kappa_2$:
\begin{equation}
\|\pi^k_\Bc v \|_\Ac \leq (1 + \kappa_1 + \kappa_1 \kappa_2) \|v\|_\Ac, \qquad \forall v \in \Ac^k.
\end{equation}
\end{theorem}

\begin{proof}
The first assumption implies that there exists an operator
\begin{equation}
Q_\Bc: \Ac^k \to \mathfrak{Z}^{k, \perp}(\Bc), \qquad dQ_\Bc v = P_{\mathfrak{B}(\Bc)} dv.
\end{equation}
It follows from \cref{eq: ass1} that $Q_\Bc$ is bounded:
\begin{equation}
\|Q_\Bc v\|_\Ac \leq \kappa_1 \|d Q_\Bc v\|  = \kappa_1\| P_{\mathfrak{B}(\Bc)} dv \| \leq \kappa_1
\|v\|_\Ac.
\end{equation}

Next, we consider the orthogonal projection $P_\Hf: \Ac^k \to \Hf^k$. When restricting the domain of the projection to $\Hf(\Bc)$, it defines a map $P_\Hf|_{\Hf(\Bc)}: \Hf^k(\Bc) \to \Hf^k$. We denote the inverse of this map by $R_\Bc: \Hf \to \Hf(\Bc)$, which exists and is bounded because of \cref{eq: ass2}.

We now consider the following map:
\begin{equation}\label{eq: bounded cochain projection Q and R}
\pi_\Bc^k = P_{\mathfrak{B}(\Bc)} + R_\Bc P_{\mathfrak{H}}(I - Q_\Bc) + Q_\Bc. 
\end{equation}
We claim that \cref{eq: bounded cochain projection Q and R} defines a bounded cochain projection. For this to be true, we need to verify the following three properties for \cref{eq: bounded cochain projection Q and R}: 
\begin{enumerate}
    \item $\pi_\Bc^k$ is a cochain map
    \item $\pi_\Bc^k$ is bounded
    \item $\pi_\Bc^k$ is a projection
\end{enumerate}
Firstly, for $\pi^k_\Bc$ to be a cochain map it satisfies $d\pi^k_\Bc v = \pi^{k+1}_\Bc dv$. Using the Hodge decomposition, we can decompose $v = du + q + d^*w$.

On one hand, 
\begin{equation}
d\pi_\Bc^{k+1} v = dP_{\Bf_\Bc} v + d(R_\Bc P_\Hf (I - Q_\Bc ) v) + dQ_\Bc v.
\end{equation}
The first two terms are zero, since $P_{\Bf_\Bc} v = du_\Bc$ and thus $dP_{\Bf_\Bc} v= ddu_\Bc = 0$, and similarly $dq_\Bc = 0$.

On the other hand, 
\begin{equation}
\pi_\Bc^k dv = P_{\Bf_\Bc} dv + (R_\Bc P_\Hf (I - Q_\Bc) dv) + Q_\Bc dv.
\end{equation}
In this expression, the second term and the third term are zero. The second term is zero since $P_\Hf dv = 0$, while the last term is zero since $Q_\Bc: \Ac^k \to \Zf^{k,\perp}(\Bc)$, but $dv \in \Bf^k \subset \Zf^k$, hence $Q_h dv = 0$. By definition of $Q_\Bc$, $dQ_\Bc v = P_{\Bf_\Bc} dv$. The only non-zero terms coincide, and thus we conclude that $\pi_\Bc^{k+1} dv = d\pi^k_\Bc v$.

Secondly, $\pi^k_\Bc$ is bounded because both $Q_\Bc$ and $R_\Bc$ are bounded, and each orthogonal projection $P_K$ is bounded with bound 1.

\begin{subequations}
\begin{align}
\| \pi_\Bc^k v\|_\Ac &= \|(P_{\mathfrak{B}(\Bc)} + R_\Bc P_{\mathfrak{H}}(I - Q_\Bc) + Q_\Bc)v \|_\Ac \\
&\leq  \| P_{\mathfrak{B}(\Bc)}v \|_\Ac + \| R_\Bc P_{\mathfrak{H}}(I - Q_\Bc)v \|_\Ac + \| Q_\Bc v \|_\Ac \\ 
&\leq \|v\|_\Ac + \kappa_1 \kappa_2 \|v\|_\Ac + \kappa_1 \|v\|_\Ac = (1 + \kappa_1 + \kappa_1 \kappa_2) \|v\|_\Ac.
\end{align}
\end{subequations}

If $v_\Bc = du_\Bc + q_\Bc + d^*_\Bc w_\Bc \in \Bc^k$, then $Q_\Bc v_\Bc = P_{\Bf^*(\Bc)} v_\Bc$ and $R_\Bc P_{\Hf} v_\Bc = P_{\Hf(\Bc)} v_\Bc$. From this we see that the map $\pi_\Bc^k$ is invariant on $v_\Bc \in \Bc^k$: 
\begin{equation}
\pi_\Bc^k v_\Bc = du_\Bc + q_\Bc + d^*_\Bc w_\Bc = v_\Bc.
\end{equation}
This invariance proves the final claim that $\pi^k$ is a projection. We can therefore conclude that $\pi^k$ as defined in \cref{eq: bounded cochain projection Q and R} is a bounded cochain projection.
\end{proof}

\Cref{theorem: existence of bcp} guarantees the existence of bounded cochain projections, given the assumptions \cref{eq: ass1} and \cref{eq: ass2}.

\subsection{Validity of assumptions} 
The previous section establishes the existence of bounded cochain projections, under the assumptions \cref{eq: ass1} and \cref{eq: ass2}. We will briefly discuss the validity of the assumptions in \cref{eq: ass1} and \cref{eq: ass2} for the \v{C}ech-de Rham complex $(\Cc^\bullet, D)$ and the subcomplex $(\Ec^\bullet, D)$ induced by the cochain map $\Xi^\bullet: \Sc^\bullet \to \Cc^\bullet$.

The cochain map $\Xi^\bullet$ is bounded from above and below: for each $k \in \{0,..., n\}$, there exists positive constants $C_1,C_2$ such that the following inequality holds:
\begin{equation}\label{eq: xi bounded}
C_1 \|v\|_\Sc \leq \|\Xi^k(v)\|_\Cc \leq C_2 \|v\|_\Sc.
\end{equation}

Since the simplicial de Rham complex is a closed Hilbert complex, it admits a Poincaré inequality:
\begin{equation} \label{eq: simplicial poincare inequality}
\|v\| \leq c_{\Sc} \|D_\Sc v\|, \qquad \forall v \in \mathfrak{Z}^{k,\perp}(\Sc).
\end{equation}
Using the inequality in \cref{eq: xi bounded} together with the Poincaré inequality \cref{eq: simplicial poincare inequality}, we get a Poincaré inequality for the subcomplex $\mathcal{E}^k \subset \Cc^k$:
\begin{equation}
\|\Xi^k(v)\| \leq C_2 \|v\| \leq C_2 c_{\Sc} \|D_\Sc v\| \leq \frac{C_2}{C_1} c_{\Sc} \|\Xi^{k+1}(D_\Sc v)\| = \frac{C_2}{C_1} c_{\Sc}  \|D_\Ac \Xi^k(v)\|. 
\end{equation}

It is shown\footnote{The authors postulate that the result holds for any positive integer $n \in \mathbb{N}$} for $n \leq 3$, where $n$ is the maximal degree of the Hilbert complex (which coincides with the maximal dimension of the mixed-dimensional hierarchy), that the mixed-dimensional de Rham complex is quasi-isomorphic to the de Rham complex \cite[Theorem 3.3]{mdG}.

The \v{C}ech-de Rham complex has cohomology which is isomorphic to the de Rham complex, hence the two complexes $\Sc^\bullet$ and $\Cc^\bullet$ have the same cohomology. Moreover, the cochain map $\Xi^\bullet$ induces an isomorphism in cohomology between the mixed-dimensional de Rham complex $\Sc^\bullet$ and its image $\Ec^\bullet$, hence the two Hilbert complexes of interest have the same cohomology, and the assumption \cref{eq: ass2} holds.

We conclude that the assumptions \cref{eq: ass1} and \cref{eq: ass2} are satisfied for the Hilbert complexes $(\Cc^\bullet, D)$ and $(\Ec^\bullet, D)$, thus \cref{theorem: existence of bcp} is valid.

\section{Error estimates}\label{section: error estimates}
Now that we have established both the subcomplex property and the existence of bounded cochain projections, we proceed by formulating \emph{a priori} error estimates and \emph{a posteriori} error estimates between solutions in the full \v{C}ech-de Rham complex and the subspace which is the embedded mixed-dimensional de Rham complex. These error estimates compare the solution to Hodge-Laplace problems on mixed-dimensional geometries to the solution of an equivalent Hodge-Laplace problem on an equidimensional (overlapping) geometry. 

\subsection{Poincaré constants for mixed-dimensional de Rham complex}
The error estimates we produce are dependent on Poincaré constants of the different Hilbert complexes. In this section we will review and discuss the different Poincaré constants that arise in this setting.

The classical Poincaré inequality is usually stated in terms of gradient, and allows us to bound a function by its derivative:
\begin{equation} \label{eq: classical PC ineq}
\|u\|_{L^2(\Omega)} \leq c_P \|\nabla u\|_{L^2(\Omega)}, \qquad \forall u \in H_0^1(\Omega).
\end{equation}
%For a convex domain $\Omega$, the Poincaré constant $c_P$ has an upper bound $\pi^{-1} \mathrm{diam}(\Omega)$.

The classical Poincaré inequality can be generalized in the following way: any closed Hilbert complex $(\Ac^\bullet, d)$ permits a Poincaré-type inequality for each $k$;
\begin{equation}\label{eq: poincare ineq HC}
\|u\| \leq c_{\Ac,k} \|du\|, \qquad u \in \Ac^k \cap \Zf^{k,\perp}.
\end{equation}
Under the assumption that the Hilbert complex is closed, the subspace $\Bf^k$ is a Hilbert space for each $k$, and $d^k$ defines a bounded linear isomorphism $d^k: \Zf^{k,\perp} \to \Bf^k$, which admits a bounded inverse. The Poincaré constant $c_{\Ac,k}$ is the bound for the inverse of $d^k$.

The Poincaré constant can be described by the following infimum:
\begin{equation}\label{eq: pc full}
\frac{1}{c_{\Ac, k}} = \inf_{u \perp \ker D^k } \frac{\|D^ku\|}{\|u\|}.
\end{equation}

We can also consider the Poincaré constant we get when considering a subspace (or more generally a subcomplex) $V^k \subset \Ac^k$:
\begin{equation}\label{eq: pc subspace}
\frac{1}{c_{V,k}} = \inf_{u \in V^k \perp \ker D^k} \frac{\|D^ku\|}{\|u\|}.
\end{equation}
A classical result \cite{payne1960optimal} shows that the Poincaré constant of the gradient operator is bounded by $\pi^{-1} \mathrm{diam}(\Omega)$ for any convex domain $\Omega$. The generalization of this result tells us that all the Poincaré constants in the de Rham complex have the bound $c_k \leq \pi^{-1} \mathrm{diam}(\Omega)$ \cite{pauly2020poincare}.

We also make use of the following result which equates the Poincaré constant of a differential operator $D^k$ with the Poincaré constant associated with its adjoint $D^*_{k+1}$:
\begin{proposition}
\begin{equation}
\frac{1}{c_{\Ac,k}} = \inf_{u \perp \ker D^k} \frac{\|D^ku\|}{\|u\|} =\frac{1}{c_{\Ac, k+1}^*} =  \inf_{v \perp \ker D^*} \frac{\|D^*_{k+1}u\|}{\|u\|}.
\end{equation}
\begin{proof}
See \cite[Lemma 2.2]{pauly2020solution}.
\end{proof}
\end{proposition}

Another type of Poincaré inequality exists when dealing with functions that are not necessarily vanishing on the boundary (and more generally, differential forms that are not orthogonal to the kernel $\Zf^{k}$).

We write $u_\Omega = \frac{1}{|\Omega|}\int_\Omega u$ for the average value of $u \in H\Lambda^k(\Omega)$. Then, we have a Poincaré inequality for any $u \in H\Lambda^k(\Omega)$:
\begin{equation}
\|u - u_\Omega\|\leq c_{k,\Omega} \|d^k u\|.
\end{equation}
The \v{C}ech-de Rham complex, the simplicial de Rham complex and the embedded simplicial de Rham complex all permit a global Poincaré-type inequality as \cref{eq: poincare ineq HC}. As shown earlier, the Poincaré constant for the subcomplex $\Ec^\bullet$, can be expressed in terms of the Poincaré constant for the simplicial de Rham complex $\Sc^\bullet$: 
\begin{equation}
c_{\Ec,k} = \frac{C_2}{C_1} c_{\Sc,k},
\end{equation}
where $C_1$ and $C_2$ are bounds on the embedding map \cref{eq: embedding}. Since the Hilbert complexes we consider are product spaces of the de Rham complex, there are also local Poincaré-type inequalities for each $\Omega_i$ and $U_i$.

\subsection{A priori error estimates}
Now that we have established the embedded simplicial de Rham complex $(\Ec^\bullet, D)$ as a subcomplex of the \v{C}ech-de Rham complex through the embedding $\Xi^\bullet: \Sc^\bullet \to \Cc^\bullet$, and the existence of bounded cochain projections from the full \v{C}ech-de Rham complex onto the subcomplex, we can establish general a priori error estimates as presented in \cite{arnoldFEEC}.

\begin{theorem}[Error estimates for Hodge-Laplace problems]\label{theorem: a priori}
Let $(u, v, q) \in \Cc^{k-1} \times \Cc^{k} \times \Hf^k(\Cc)$ solve the mixed weak formulation of the Hodge-Laplace problem on $(\Cc^\bullet, D)$ and let $(u_\Ec, v_\Ec, q_\Ec) \in \Ec^{k-1} \times \Ec^k \times \Hf^k(\Ec)$ solve the mixed weak formulation of the Hodge-Laplace problem on $(\Ec^\bullet, D_\Ec)$. Then
\begin{subequations}
\label{eq: a priori error estimates}
\begin{align}
\|u - u_\Ec\|_{\Cc} &+ \|v - v_\Ec\|_{\Cc} + \|q - q_\Ec\| \leq \\
 \mathfrak{c} \Big( \inf_{\hat{u} \in \Ec^{k-1}} \|u - \hat{u}\|_\Cc &+ \inf_{\hat{v} \in \Ec^{k}} \|v - \hat{v}\|_\Cc + \inf_{\hat{q} \in \Ec^k} \|q - \hat{q}\|_\Cc  + M_k \inf_{\hat{v} \in \Ec^k} \|P_{\mathfrak{B}}^k v - \hat{v}\|_\Cc \Big),
\end{align}
\end{subequations}
where $\mathfrak{c}$ depends on the Poincaré constant $c_{\Cc,k}$ and $\|\pi^k\|$ (where $\pi^k: \Cc^k \to \Ec^k$ denotes the bounded cochain projection), while the other constant is given by
\begin{equation}\label{eq: bound in harmonic forms}
M_k = \sup_{\substack{\hat{q} \in \Hf^k(\Cc), \\ \|\hat{q} \|=1}} \|(\mathrm{id} - \pi^k) \hat{q} \|. 
\end{equation}
\end{theorem}
We refer to \cite[Section 5.2]{arnoldFEEC} for the proof, where the theorem is proven for an arbitrary closed Hilbert complex $(\Ac^\bullet,d)$, a subcomplex $\Bc^\bullet \subset \Ac^\bullet$, and a bounded cochain projection $\pi^\bullet: \Ac^\bullet \to \Bc^\bullet$. One can obtain improved error estimates if we have a cochain projection that is bounded in $L^2$ norm instead of the graph norm. This is shown in \cite{arnoldFEEC3}. 

For a simply connected domain, the cohomology $\Hf^k(\Cc)$ is zero for $k>0$, and thus the constants $M_k$, as well as the terms $\|q - q_\Ec\|$ and $\inf_{\hat{q} \in \Ec^k }\|q - \hat{q}\|_\Ac$ vanish. Thus, \cref{eq: a priori error estimates} for a $k \neq 0$ Hodge-Laplacian on a simply-connected domain takes a much simpler expression:
\begin{equation}
\|u - u_\Ec \|_{\Cc} + \|v - v_\Ec\|_{\Cc}  \leq c \Big( \inf_{\hat{u} \in \Ec^{k-1}} \|u - \hat{u}\|_\Cc + \inf_{\hat{v} \in \Ec^{k}} \|v - \hat{v}\|_\Cc \Big).
\end{equation}

In finite element exterior calculus, one obtains an approximation property based on the mesh size $h$ for the triangulation of the domain. In particular, one aims to show that in the limit as $h \to 0$, the infimum tends towards zero: $\inf\limits_{u_h \in V_h} \|u - u_h\|_V \to 0$.  

We want to express a similar type of approximation property, but instead of considering the limit as the mesh size approaches zero, we consider a family of open covers $\{\Uc^\epsilon\}_{\epsilon \in (0,\epsilon_M)}$ and take the limit as the maximal diameter of overlaps goes towards zero, as illustrated in \cref{fig: parameterized open cover}.

\begin{figure}[htb]
    \centering
    \includegraphics[width=1.0\linewidth]{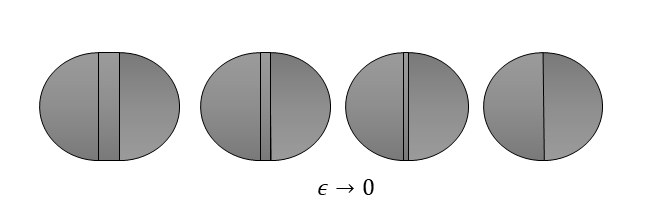}
    \caption{Three left domains: A simple open cover consisting of two open sets, with diameter in the transversal direction equal to a parameter $\epsilon$. Rightmost domain: When $\epsilon \to 0$, we get a mixed-dimensional geometry.}
    \label{fig: parameterized open cover}
\end{figure}

\begin{proposition}\label{proposition: liminf}
Let $\Uc_{\epsilon}$ be a family of \v{C}ech cover associated to a simplicial geometry, where the subscript $\epsilon$ denotes the maximal diameter in the transversal direction.

Given $u \in \Cc^k$, 
\begin{equation}\label{eq: liminf}
%\lim\limits_{\epsilon \to 0} 
\inf\limits_{u_\Ec \in \Ec^k} \|u - u_\Ec \|_\Cc = \mathcal{O}(\epsilon).
\end{equation}
\end{proposition}

\begin{proof}
We consider general case when the inner products for the \v{C}ech-de Rham complex are weighted by the parameter $\epsilon$, as in \cref{eq: weighted cdr inner products}. Unweighted inner products for the \v{C}ech-de Rham complex can be considered a special case, and the result for unweighted inner products follows immediately. 

For the $\epsilon$-weighted inner products, the expression in \cref{eq: liminf} becomes the following:
\begin{equation}
 \|u - u_\Ec \|_\Cc = \sum_{p+q=k} \sum_{i \in \Ic^p} \epsilon^{-p/2} \|u_i - u_{\Ec,i}\|_D.
\end{equation}
The recursive norm for a differential form of degree $(p,q)$ will give contributions to subdomains corresponding to $(p+1,q)$. Thus, some domains will have contributions multiple (but finitely many) times. We will show the desired result for a fixed $(p,q)$, and ignore the recursive norm.

For $p=0$, we have unit weights ($\epsilon^{-0}=1$), and since the subcomplex imposes no restriction on the differential forms on the non-overlapping parts of $U_i$, the difference 
$u_i - u_{\Ec,i}$ vanishes on $\Tilde{U}_i$. Moreover, the difference on the overlapping parts of $U_i$ goes towards zero in the limit, since the measure of the overlaps tends towards zero with order $\epsilon^{p/2}$.

For $p \neq 0$, we have a weight $\epsilon^{-p/2}$, and we are considering domains $U_i$, $i \in \Ic^p$ with $p$ transversal directions of length $\epsilon$. We can set $u_\Ec$ to be equal to the mean value of $u$:
\begin{equation}
u_{\Ec,i}= \frac{1}{|U_i|}\int_{U_i} u_i.
\end{equation}

Using the Poincaré inequality,
\begin{equation}\label{eq: proposition inequality}
\epsilon^{-p/2} \|u_i - u_{\Ec,i}\|_D \leq \epsilon^{-p/2} c_{\Ac,k} \|du_i\|.
\end{equation}
Since the difference $u_i - u_{\Ec,i}$ is bounded in the tangential direction, the Poincaré constant is proportional to $\epsilon$ (and $c_{\Ac,k} \leq \frac{\sqrt{n}\epsilon}{\pi}$). Each $u_i \in H\Lambda^q(U_i)$, hence the terms $\|du_i\|$ are finite, and of order $\epsilon^{p/2}$.

The Poincaré constant together with the term $\|du_i\|$ has order $\epsilon^{p/2+1}$, and accounting for the weight $\epsilon^{-p/2}$, we still have a right-hand side in \cref{eq: proposition inequality} that is of order $\epsilon$. We therefore conclude that 
$\lim\limits_{\epsilon \to 0} \inf\limits_{u_\Ec \in \Ec^k} \|u - u_\Ec \|_\Cc = 0$ with order $\epsilon$.
\end{proof}

\subsection{A posteriori error estimates} \label{subsection: a posteriori}
In \cite{pauly2020solution}, the author analyses general solution theory and a posteriori error estimates to first order systems of the type
\begin{subequations} \label{eq: 1st order system PAULY}
\begin{align} 
D^ku &= f, \\
D_k^*u &= g, \\
\pi_\Hf u &= h.
\end{align}
\end{subequations}

The framework is further extended to second order systems of the type
\begin{subequations} \label{eq: 2nd order system 1}
\begin{align} 
D^*_{k+1} D^ku &= f, \\
D_k^*u &= g, \\
\pi_\Hf u &= h.
\end{align}
\end{subequations}

By writing $D^k u = v$, we can write \cref{eq: 2nd order system 1} as a first order system of six equations:
\begin{subequations}
\begin{align}
D^ku &= v, \qquad & D^{k+1} v &= 0, \\
D_k^*u &= g, \qquad &D^*_{k+1} v &= f, \\
\pi_\Hf u &= h, \qquad &\pi_\Hf v &= 0. \\
\end{align}
\end{subequations}
Thus, a posteriori error estimates for a first order system like \cref{eq: 1st order system PAULY} can be extended to a system like \cref{eq: 2nd order system 1}. Likewise, a second order system of the type
\begin{subequations}  \label{eq: second order system 2}
\begin{align} 
D^*_{k+1} D^ku &= f, \\
D^{k-1}D_k^*u &= g, \\
\pi_\Hf u &= h,
\end{align}
\end{subequations}
can be written as a first order system
\begin{subequations} \label{eq: second order system 2 firstO}
\begin{align} 
D^ku &= v, \qquad &D^{k+1} v &= 0, \qquad &D^{k-1}q &= g, \\
D_k^*u &= q, \qquad &D^*_{k+1} v &= f, \qquad &D^*_{k-1}q &= 0, \\
\pi_\Hf u &= h, \qquad &\pi_\Hf v &= 0, \qquad &\pi_\Hf q &= 0. \end{align}
\end{subequations}

The Hodge-Laplacian for $k=0$ can readily be identified with \cref{eq: 2nd order system 1}, with $g = 0$ and $h = 0$. Similarly, the Hodge-Laplacian for $k=n$ can be identified with the similar system, but with the roles of $D$ and $D^*$ reversed. If the given initial data to the Hodge-Laplace problem is either entirely in $\Bf^k$ or $\Bf^*_k$, then the same identifications can be made for Hodge-Laplace problems for arbitrary $k$.

For a general $k$-Hodge-Laplace problem, we can identify it with \cref{eq: second order system 2}, where $f$ and $g$ is the Hodge decomposition of the initial data for the Hodge-Laplace problem, and $h =0$ as before.

The a posteriori error estimates we present herein are a special case of the ones presented in \cite{pauly2020solution}. In order to apply the a posteriori theory presented in their work, we require the Hodge-Laplace problem to satisfy one of the following:
\begin{enumerate}
    \item The Hodge-Laplace problem is a $k=0$ problem
    \item The Hodge-Laplace problem is a $k=n$ problem
    \item The Hodge-Laplace problem is a $\Bf^*$-problem
    \item The Hodge-Laplace problem is a $\Bf$-problem
    \item The Hodge decomposition of the right-hand side is known
\end{enumerate}

We now state the results from \cite[Theorem 4.7]{pauly2020solution} applied to our framework when assumption (3) above is satisfied, e.g. for a $\Bf^*$-problem, with the understanding that similar estimates can be obtained for $\Bf$-problems by reversing the roles of $D$ and $D^*$, and estimates for general $k$-Hodge-Laplace problems can be obtained whenever the Hodge decomposition is known by considering the system in \cref{eq: second order system 2 firstO}.

\begin{theorem}\label{theorem: a posteriori}
Let $(u, v) \in \Cc^k$ be an exact solution to the Hodge-Laplace $\Bf^*$-problem on the \v{C}ech-de Rham complex, and let $(u_\Ec, v_\Ec)$ be an approximation of $(u, v)$ derived from the Hodge-Laplace $\Bf^*$-problem on $\Sc^k$.

The Hodge-Laplace $\Bf^*$-problem can be identified with a first order system 
\begin{subequations}
\begin{align}
D^k u = v, \qquad  D^{k+1}v=0,\\
D^*_{k} u = 0, \qquad  D_{k+1}^*v = f, \\
\pi^k_\Hf u =0, \qquad \pi^{k+1}_\Hf v=0.
\end{align}
\end{subequations}

Then the errors $e^u = u - u_\Ec$ and $e^v = v - v_\Ec$ decompose according to the Hodge decomposition:
\begin{align}
 e^u &= e^u_\Bf + e^u_\Hf + e^u_{\Bf^*} \in \Bf^k \oplus \Hf^k \oplus \Bf^*_k,   \\
 \|e^u\|^2_{\Cc^k} &= \| e^u_\Bf\|^2_{\Cc^k} + \|e^u_\Hf\|^2_{\Cc^k} + \|e^u_{\Bf^*}\|^2_{\Cc^k},
\end{align}
and 
\begin{align}
 e^v &= e^v_\Bf + e^v_\Hf + e^v_{\Bf^*} \in \Bf^{k+1} \oplus \Hf^{k+1} \oplus \Bf^*_{k+1},   \\
 \|e^v\|^2_{\Cc^{k+1}} &= \| e^v_\Bf\|^2_{\Cc^{k+1}} + \|e^v_\Hf\|^2_{\Cc^{k+1}} + \|e^v_{\Bf^*}\|^2_{\Cc^{k+1}},
\end{align}

The following a posteriori error estimates for the respective error parts hold:
\begin{enumerate}
\item 
\begin{equation}
\|e^u_{\Bf^*}\|^2_{\Cc^k} \leq \left( c_{k}^2 \|D^*_{k+1} D^k u_\Ec - f\|_{\Cc^{k}} \right)^2.
\end{equation}
\item 
\begin{equation}\label{eq: error of u_H}
\|e^u_{\Hf}\|^2_{\Cc^k} = \min_{\varphi \in \dom(D^{k-1})} \min_{\phi \in \dom(D^*_{k+1})} \| D^{k-1}\varphi + D^*_{k+1}\phi - u_\Ec \|_{\Cc^k}^2.
\end{equation}
\item 
\begin{equation}
\|e^v_\Bf\|_{\Cc^{k+1}}^2 \leq \left( c_k \|D^*_{k+1}v_\Ec - f\|_{\Cc^k} \right)^2.
\end{equation}
\end{enumerate}
For the $\Bf^*$-problem the error $e^u_\Bf = 0$ since $u \in \Bf^*_k$.
Moreover, for any Hodge-Laplace problem, $e^v_\Hf$ and $e^v_{\Bf^*}$ are zero because $e^v = v - v_\Ec = Du - Du_\Ec \in \Bf^{k+1}$, hence orthogonal to both $\Bf^*_{k+1}$ and $\Hf^{k+1}$.
\end{theorem}

The error in harmonic form can be estimated by the gap
\begin{equation}
\mathrm{gap}(\Hf^k(A),\Hf^k(B)) \leq M_k,
\end{equation}
where $M_k$ is the constant which appears in \cref{eq: bound in harmonic forms}. 

\begin{remark}
Note that the expressions depend on the Poincaré constant $c_k$. Since the Poincaré constant is derived as the infimum over the function space, if we restrict to a smaller subcomplex $V^k \subset A^k$, we have that $\frac{1}{c_k} \leq \frac{1}{c_{V,k}}$, and therefore $c_{V,k} \leq c_k$. We observe that a smaller subspace (or more generally subcomplex) will yield a smaller (hence better) Poincaré constant, and in turn better error estimates. Such subspaces are discussed in \cite{varela2023posteriori}; where the authors consider three cases: no conservation (full complex), subdomain mass conservation (which allows for local Poincaré constants for domains $U_i$) and exact mass conservation (where any constant is permissible).
\end{remark}

%\begin{proposition}
%For the $k=0$ Hodge-Laplace problem, $\|e_\Bf\|$ and $\|e_{\Hf}\|$ vanishes.
%\end{proposition}
%\begin{proof}
%Since we are at the left end of the complex, where $\Bf^0 = 0$ for both $\Cc$ and $\Ec$, hence $\|e_\Bf\| = 0$. 
%By the same argument, $\varphi = 0$ in \cref{eq: error of u_H}. We are left with 
%\begin{equation}
% \min_{\phi \in \dom(D^*_{k+1})} \|D^*_{k+1}\phi  -u_\Ec \|_{\Cc^k}^2.
%\end{equation}
%\end{proof}
%
%
%
%
%
%
%
%
%
%
%
%
%
%
%
%

\section{Examples}\label{sec:examples}
\subsection{One-dimensional elastically joined rods}
In this section we consider two elastic rods that are joined together. We simplify the example considerably by treating the rods as one-dimensional. We consider two different geometrical descriptions for the mathematical models: an equidimensional model where the overlap of the two rods have non-zero length, and the mixed-dimensional model wherein the overlap is represented by a zero-dimensional interface. Such models for one-dimensional elasticity appear in literature, see e.g. \cite{mejri2021carleman}.

Displacement and stress of the joined rods are governed by the Hodge-Laplace problem. We will compare the two associated problems: the Hodge-Laplace problem on the full \v{C}ech-de Rham complex and the Hodge-Laplace problem on the simplicial de Rham complex, embedded by the cochain map constructed in \cite{holmen2024injective}.

\begin{figure}[htb]
    \centering
    \includegraphics[width=1.0\linewidth]{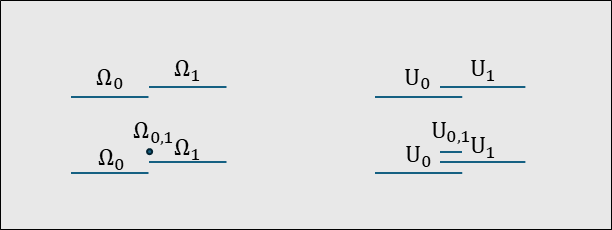}
    \caption{The mixed-dimensional geometry consists of two adjacent intervals, while the \v{C}ech open cover consists of two overlapping intervals.
    The Hodge-Laplace problem for $k=0$ is considered in this example, and is defined on the two adjacent/overlapping intervals (top left / top right). While we do elaborate the Hodge-Laplace problem for $k=1$, it is instructive to note the difference in geometrical structure. I.e. the $k=1$ problem  is defined on both the adjacent/overlapping interval, as well as the common boundary point/overlap (bottom left / bottom right). The intervals are separated vertically for visual clarity.
    }
    \label{fig: 1d geometry}
\end{figure}

We write $\Omega_0 = (-1,0)$ and $\Omega_1 = (0,1)$, and $\Omega_{0,1} = \{0\}$. The open cover associated with the equidimensional model is $\Uc = \{U_0, U_1\}$, where $U_0 = (-1, \epsilon)$, $U_1 = (-\epsilon, 1)$ and $U_{0,1} = U_0 \cap U_1 = (-\epsilon, \epsilon)$. While the geometrical description for this example is simple, it covers all interesting aspects of these types of coupled problems that occur in dimension $n=1$. 

The two Hodge-Laplace problems can be written as $D^{*,w}D u = f$ for $k=0$ and $DD^{*,w}v=g$ for $k=1$. 
The right-hand side of the Hodge-Laplace problem represents body forces acting on the rods, and the weights that appear in the inner products (and thus appear in the definition of the adjoint operator) correspond to material parameters.  

For differential $k$-forms $u,v \in L^2\Lambda^k_{w}(U_i)$ with function components $u_j, v_j \in L^2(U_i)$ we define the weighted inner product on $U_i$:
\begin{equation}
\langle u, v \rangle_{L^2\Lambda^k_{w}(U_i)} = \int_{U_i} \sum_{j \in \mathcal{N}^k} w_{k,i} u_j v_j \; \mathrm{Vol}(U_i).
\end{equation}
Similarly for the mixed-dimensional representation, we have weighted inner products on $\Omega_i$:
\begin{equation}
\langle u, v \rangle_{L^2\Lambda^k_{w}(\Omega_i)} = \int_{\Omega_i} \sum_{j \in \mathcal{N}^k} w_{k,i} u_j v_j \; \mathrm{Vol}(\Omega_i).
\end{equation}

We consider \cref{eq: Hodge-Laplace problem} on the geometry shown in \cref{fig: 1d geometry}. In one spatial dimension, the exterior derivative is identified with the usual derivative with $H^1$ function spaces as its domains, and the codifferential is identified with the usual derivative with a change of sign, with domain $H^1_0$. Since we consider weighted inner products, the codifferential operator picks up the weights:
\begin{equation}
d^{*,w}_k= w_{k-1}^{-1} d^{*,1}_k w_k.
\end{equation}

The jump operator $\delta: H^1(U_0)\times H^1(U_1) \to H^1(U_{0,1})$ simply takes the difference of $u_0 \in H^1(U_0)$ and $u_1 \in H^1(U_1)$, restricted to $U_{0,1}$ where they both are defined.

By definition, the adjoint of the difference operator satisfies the following equation:
\begin{subequations}
\begin{align}
\langle \delta u, v \rangle &= \langle u, \delta^*v \rangle,  \\
\int_{U_{0,1}} (u_1| - u_0|)v &= \int_{U_0} u_0(\delta^*_0v) + \int_{U_1} u_1(\delta^*_0v), \\
\int_{U_{0,1}} u_1|v -\int_{U_{0,1}} u_0|v &= \int_{U_0} u_0(\delta^*_0v) + \int_{U_1} u_1(\delta^*_0v).
\end{align}  
\end{subequations}
With weighted inner products, the adjoint of the difference operator is defined as $\delta^{*,w}_i p_{0,1}= w_{i}^{-1} (-1)^{i+1} \mathbbm{1}_i (w_{0,1} p_{0,1})$, where $\mathbbm{1}$ denotes the characteristic function:
\begin{equation}
\mathbbm{1}_i = \begin{cases}
1, \qquad \text{on} \; U_i \\
0, \qquad \text{otherwise}.
\end{cases}
\end{equation}

The diagram for the \v{C}ech-de Rham complex to the open cover $\Uc=\{U_0, U_1\}$ consists of the one-dimensional de Rham complex in the first column, as well as the jump operator taking the difference between two functions $u_1 \in H^1(U_1)$ and $u_0 \in H^1(U_0)$ and mapping the difference onto the $L^2$ space of the overlap $U_{0,1}$:
\begin{equation}
% https://q.uiver.app/#q=WzAsMyxbMCwyLCJIXjEoVV8wKSBcXHRpbWVzIEheMShVXzEpIl0sWzIsMiwiTF4yKFVfezAsMX0pICJdLFswLDAsIkxeMihVXzApIFxcdGltZXMgTF4yKFVfMSkiXSxbMCwxLCJcXGRlbHRhIiwxXSxbMCwyLCJkIiwxXV0=
\begin{tikzcd}
	{L^2(U_0) \times L^2(U_1)} \\
	\\
	{H^1(U_0) \times H^1(U_1)} && {L^2(U_{0,1}) }
	\arrow["d"{description}, from=3-1, to=1-1]
	\arrow["\delta"{description}, from=3-1, to=3-3]
\end{tikzcd}
\end{equation}

\begin{example}    
We consider the \v{C}ech-de Rham complex associated to \cref{fig: 1d geometry} with inner products of $\Ac^1$ weighted by $w = (w_0, w_1, w_{0,1}/\epsilon)$. The primal formulation of the $0$-Hodge-Laplacian is 
\begin{subequations}
\begin{align}    
-\frac{d}{dx}(w_0 \frac{d}{dx}u_0) - \frac{w_{0,1}}{\epsilon}\delta^{*}(u_1 - u_0) &= f_0, \\
-\frac{d}{dx}(w_1 \frac{d}{dx}u_1) - \frac{w_{0,1}}{\epsilon}\delta^{*}(u_1 - u_0) &= f_1,
\end{align}
\end{subequations}
with boundary conditions
\begin{subequations}
\begin{align}
w_0 \frac{d}{dx}u_0 &= 0, \qquad x \in \{-1,  \epsilon\}, \\
w_1 \frac{d}{dx}u_1 &= 0, \qquad x \in \{-\epsilon, 1\}.
\end{align}
\end{subequations}

Moreover, orthogonality to the kernel enforces 
\begin{equation}
\int_{-1}^{\epsilon} u_0 + \int_{-\epsilon}^1 u_1 = 0.
\end{equation}

By introducing the dual variables $\sigma_i = \frac{d}{dx}u_i$ and $\sigma_{0,1} = \delta u=  (u_1 - u_0)|_{U_{0,1}}$, we get the following weak formulation of the $0$-Hodge-Laplace problem: find $(u_0, u_1, \sigma_0, \sigma_1, \sigma_{0,1}) \in \Ac ^0 \times \Ac^1$ satisfying
\begin{subequations}
\begin{align} \label{eq: 1d system}
\sigma_0 &= \frac{d}{dx}u_0, \\
\sigma_1 &= \frac{d}{dx}u_1, \\
\sigma_{0,1} &= u_1 - u_0, \\
-\frac{d}{dx}(w_0 \sigma_0) - \frac{w_{0,1}}{\epsilon}\delta_0^{*} \sigma_{0,1} &= f_0, \\
-\frac{d}{dx}(w_1 \sigma_1) - \frac{w_{0,1}}{\epsilon}\delta_1^{*} \sigma_{0,1} &= f_1, 
\end{align}
\end{subequations}
with boundary conditions
\begin{subequations}
\begin{align}\label{eq: 1d boundary conditions}
\sigma_0(-1) &= 0, \qquad \sigma_0(\epsilon) = 0, \\
\sigma_1(-\epsilon) &=0, \qquad \sigma_1(1) = 0. 
\end{align}
\end{subequations}

Equations \eqref{eq: 1d system}-\eqref{eq: 1d boundary conditions} can be solved explicitly, as shown in the appendix. 
The plots in \cref{fig: displacement plot} and \cref{fig: stress plot} show the obtained solutions for the displacement $(u_0, u_1)$ and the stress $(\sigma_0, \sigma_1, \sigma_{0,1})$, for $\epsilon = 0.2$, $r = 1$ and $w = w_{0,1}= 1$.

\begin{figure}[h]
\centering
\begin{minipage}{.5\textwidth}
  \centering
  \includegraphics[width=1.0\linewidth]{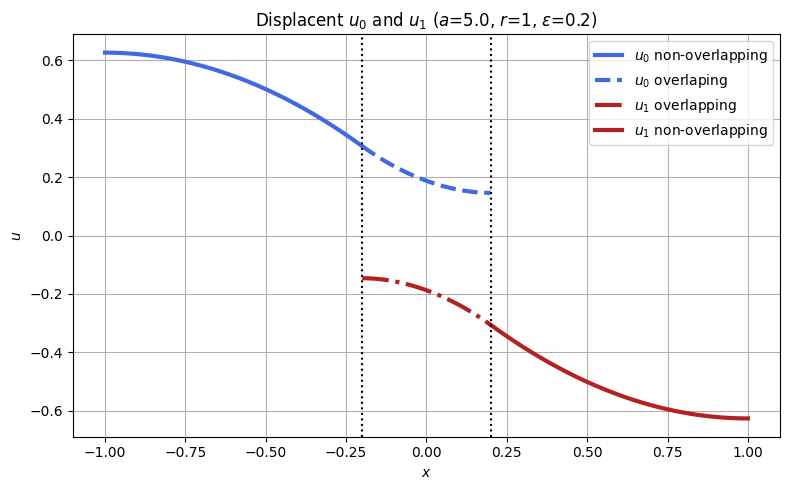}
  \caption{Displacement $u_0$ (blue) and $u_1$ (red).}
  \label{fig: displacement plot}
\end{minipage}%
\begin{minipage}{.5\textwidth}
\captionsetup{width=\linewidth}
  \centering
  \includegraphics[width=1.0\linewidth]{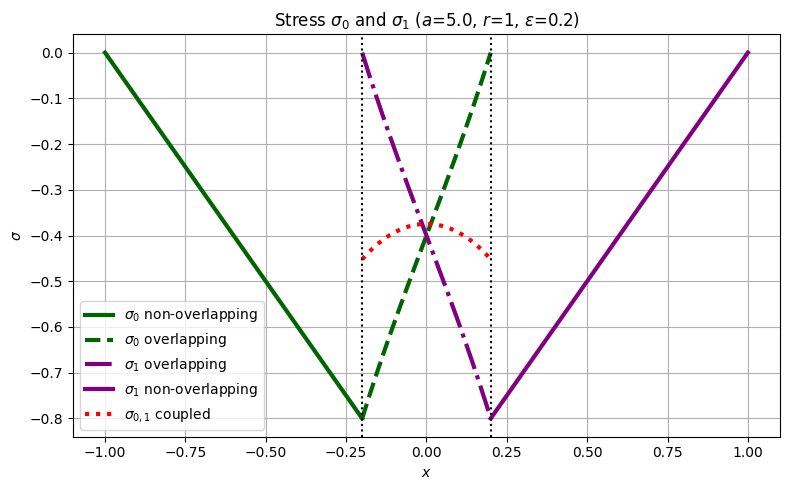}
  \caption{Stress $\sigma_0$ (green), $\sigma_1$ (purple) and coupled stress $\sigma_{0,1}$ (red).}
  \label{fig: stress plot}
\end{minipage}
\end{figure}

%\begin{figure}[htb]
%    \centering
%    \includegraphics[width=1.0\linewidth]{Figures/Fig1 - displacement upd.png}
%    \caption{Displacement $u_0$ (blue) and $u_1$ (red).
%    }
%    \label{fig: displacement plot}
%\end{figure}
%\begin{figure}[htb]
%    \centering
%    \includegraphics[width=1.0\linewidth]{Figures/Fig 2- Stress upd.png}
%    \caption{Stress $\sigma_0$ (green), $\sigma_1$ (purple) and coupled stress $\sigma_{0,1}$ (red).
%    }
%    \label{fig: stress plot}
%\end{figure}

\end{example}
\begin{example}

We now proceed by considering the same system \cref{eq: 1d system} but for the simplicial de Rham complex. Similar to the solutions for the \v{C}ech-de Rham complex, we have quadratic displacement $u_i$ and linear stress $\sigma_i$:
\begin{subequations}
\begin{align}
u_{\Sc,0}(x) &= -\frac{r}{2w}x^2 + A_{\Sc,0} x + B_{\Sc,0}, \qquad &x \in \Omega_0, \\
u_{\Sc,1}(x) &= \frac{r}{2w}x^2 +A_{\Sc,1}x + B_{\Sc,0}, \qquad &x \in \Omega_1,\\ 
\sigma_{\Sc,0}(x) &= -\frac{r}{w}x + A_{\Sc,0}, \qquad &x \in \Omega_0, \\
\sigma_{\Sc,1}(x) &= \frac{r}{w}x+A_{\Sc,1}, \qquad &x \in \Omega_1.
\end{align}
\end{subequations}

The outer boundary conditions determine the same values for $A_{\Sc,i}$ as before:
\begin{align}
A_{\Sc,0} = A_{\Sc,1} = -\frac{r}{w}.
\end{align}
The requirement that the pair $(u_0, u_1)$ is orthogonal to the constants gives $B_{\Sc,0} = -B_{\Sc,1}$, and the equations on the overlap is reduced to an algebraic constraint:
\begin{subequations}
\begin{align}    
 w\frac{d\sigma_{\Sc,i}}{dx}(0) = -r &= w_{0,1} (B_{\Sc,1} - B_{\Sc,0}) \implies B_{\Sc,1} - B_{\Sc,0} = -\frac{r}{w_{0,1}}, \\
 B_{\Sc,0} &= \frac{r}{2 w_{0,1}}, \qquad  B_{\Sc,1} = -\frac{r}{2 w_{0,1}}.
\end{align}
\end{subequations}

To realize the mixed-dimensional solution in the same spaces as the equidimensional solution, we define affine maps $\phi_0: \Tilde{U}_0 \to \Omega_0$, $\phi_1: \Tilde{U}_1 \to \Omega_1$ by
\begin{subequations}
\begin{align}
\phi_0 = \frac{(x + \epsilon)}{(1 - \epsilon)}, \qquad \phi_1 = \frac{(x - \epsilon)}{(1 - \epsilon)}. 
\end{align}
\end{subequations}

Precomposition by $\phi_0, \phi_1$ defines the pullback $\phi^*_0: L^2\Lambda^q(\Omega_0) \to L^2\Lambda^q(\Tilde{U}_0)$ and $\phi^*_1: L^2\Lambda^q(\Omega_1) \to L^2\Lambda^q(\Tilde{U}_1)$. The pullbacks are extended to the codomain $U_i$ by considering a constant extension on the overlap: $u_i$ is mapped to $\phi_{0,1}^*(\tr \: u_i)$ on $U_{0,1}$. Finally, we scale the pullbacks by $(1-\epsilon)$ to account for the dilation caused by the pullback. This defines a cochain map from the simplicial de Rham complex to the \v{C}ech-de Rham complex for the given geometries.

\begin{figure}[h]
    \centering
    \includegraphics[width=.75\linewidth]{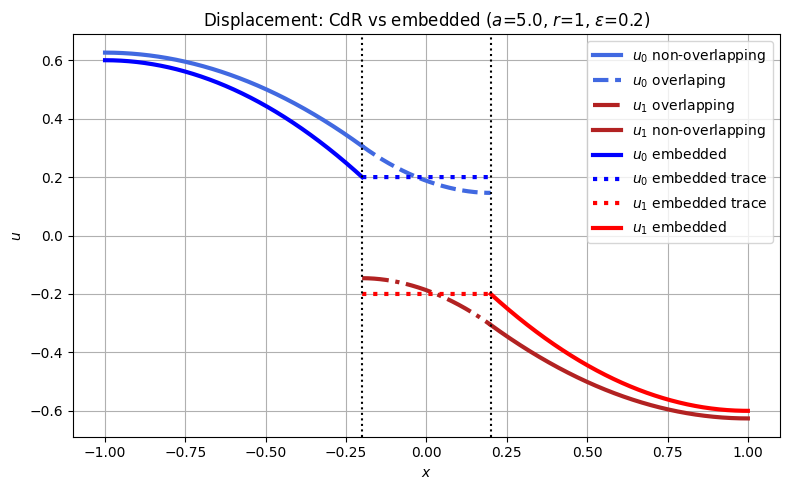}
    \caption{Comparison between embedded and \v{C}ech-de Rham displacement}
    \label{fig: displacement comparison}
\end{figure}

\begin{figure}
\centering
\begin{minipage}{.48\textwidth}
  \centering
  \captionsetup{width=\linewidth}
  \includegraphics[width=1.05\linewidth]{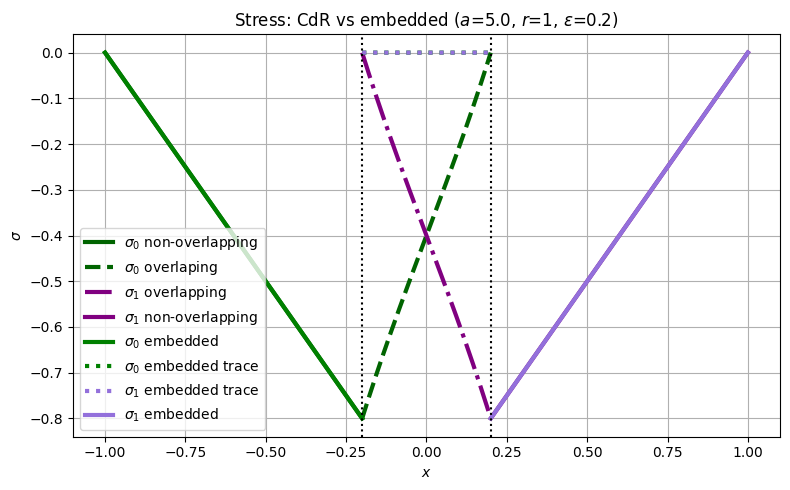}
  \caption{Comparison between embedded and \v{C}ech-de Rham stress.}
  \label{fig: stress comparison}
\end{minipage}\hfill
\begin{minipage}{.48\textwidth}
  \centering
  \captionsetup{width=\linewidth}
  \includegraphics[width=1.05\linewidth]{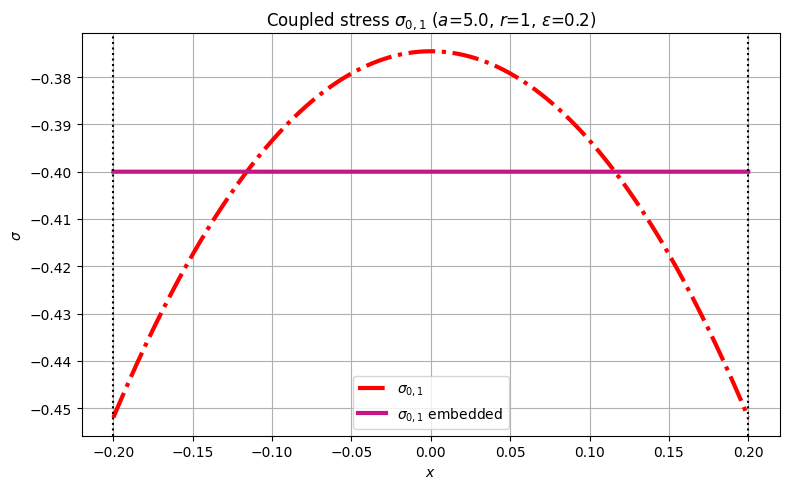}
  \caption{Comparison between embedded and \v{C}ech-de Rham coupled stress.}
    \label{fig: coupled stress comparison}
\end{minipage}
\end{figure}

%\begin{figure}[!ht]
%    \centering
%    \includegraphics[width=1.0\linewidth]{Figures/Fig6 - Stress comparison.png}
%    \caption{Comparison between embedded and \v{C}ech-de Rham stress.}
%    \label{fig: stress comparison}
%%\end{figure}
%\begin{figure}[!ht]
%    \centering
%    \includegraphics[width=1.0\linewidth]{Figures/Fig7 - Coupled stress comparison.png}
%    \caption{Comparison between embedded and \v{C}ech-de Rham coupled stress.}
%    \label{fig: coupled stress comparison}
%\end{figure}
The plots in \cref{fig: displacement comparison}, \cref{fig: stress comparison} and \cref{fig: coupled stress comparison} show the solutions for displacement, stress and coupled stress for both the \v{C}ech-de Rham Hodge-Laplace problem and the embedded mixed-dimensional Hodge-Laplace problem (with the same parameters as before).
\end{example}

\subsection{A posteriori error estimates for the example}
Now that we have both the solution to the Hodge-Laplace problem on the \v{C}ech-de Rham complex and the embedded solution of the simplicial de Rham- Hodge-Laplacian, we can compare the two solutions by relying on the theory presented in \cref{subsection: a posteriori}.

The Poincaré constant can be expressed in terms of the Rayleigh quotient,
\begin{equation}
\inf_{u \perp \ker D} R(u) = \frac{1}{C_{\Ac}^2}.
\end{equation}

where the Rayleigh quotient of $D$ is defined by
\begin{equation}
R_D(u) = \frac{\|Du\|_{\Cc^{1}}^2}{\|u\|_{\Cc^{0}}^2} = \frac{\|u_0'\|^2_{\Cc^{0,1}}+\|u_1'\|_{\Cc^{0,1}}^2+\|u_1-u_0\|^2_{\Cc^{1,0}}}{\|u_0\|^2_{\Cc^{0,0}} +\|u_1\|_{\Cc^{0,0}}^2}.
\end{equation}

We define the Rayleigh quotient associated to the independent operators $d$ and $\delta$:
\begin{subequations}
\begin{align}
R_d(u) = \frac{\|du\|_{\Cc^{0,1}}^2}{\|u\|_{\Cc^{0,0}}^2} = \frac{\|u_0'\|_{\Cc^{0,1}}^2 +\|u_1'\|_{\Cc^{0,1}}^2}{\|u_0\|_{\Cc^{0,0}}^2 +\|u_1\|^2 _{\Cc^{0,0}}}, \\
R_\delta(u) = \frac{\|\delta u\|_{\Cc^{1,0}}^2}{\|u\|_{\Cc^{0,0}}^2} = \frac{\|u_1-u_0\|^2_{\Cc^{1,0}}}{\|u_0\|^2_{\Cc^{0,0}} +\|u_1\|^2_{\Cc^{0,0}}}.
\end{align}
\end{subequations}

The infimum of the first Rayleigh quotient is realized by functions with zero global mean value, while the infimum of the second Rayleigh quotient is realized by constant pairs $(c, -c)$. Standard Poincaré inequality tells us that $1/C^2 = \frac{\pi^2}{\mathrm{diam}(U)^2}$. Since $\ker D$ are the global constant functions, the infimum is realized by the global diameter, i.e. the diameter of the union $U_0 \cup U_1$:
\begin{equation}
\inf_{u \perp \ker D} \frac{\|du\|_{\Cc^{0,1}}^2}{\|u\|_{\Cc^{0,0}}^2} = \frac{w\pi^2}{2^2}.
\end{equation}

On the other hand, the constant pair $(c, -c)$ gives the following quotient:
\begin{equation}
\frac{\|\delta(c, -c)\|_{\Cc^{1,0}}^2}{\|(c, -c)\|_{\Cc^{0,0}}^2} =\frac{\frac{w_{0,1}}{\epsilon}(2c)^2|U_{0,1}|}{|U_0|^2c^2+|U_1|^2 c^2} = \frac{8c^2}{2(1+\epsilon)c^2} = \frac{4w_{0,1}}{(1+\epsilon)}.
\end{equation}

An upper bound for the Poincaré constant can be determined based on which expression dominates: $\frac{1}{C^2} \leq \min\{R_d, R_\delta\}$. In general, the relative size of the quotients $R_d$ and $R_\delta$ depends on the values of $w$ and $w_{0,1}$. If $w/w_{0,1} \approx 1$, then the Rayleigh quotient $R_d$ will be greater, and we can therefore estimate the Poincaré constant by
\begin{equation}
C_{\Ac} \leq \frac{2}{\sqrt{w}\pi}.
\end{equation}

The error estimates for $e^u$ and $e^\sigma$ decompose in accordance with the Hodge decomposition:
\begin{subequations}
\begin{align}
u - u_\Ec &= e^u = e^u_\Bf + e^u_{\Bf^*}+ e^u_{\Hf}, \\
\|e^u\|^2_{\Cc^k} &= \|e^u_\Bf\|^2_{\Cc^k} + \|e^u_{\Bf^*}\|^2_{\Cc^k} + \|e^u_{\Hf}\|^2_{\Cc^k}, \\
\sigma - \sigma_\Ec &= e^\sigma = e^\sigma_\Bf + e^\sigma_{\Bf^*}+ e^\sigma_{\Hf}, \\
\|e^\sigma\|^2_{\Cc^{k+1}} &= \|e^\sigma_\Bf\|^2_{\Cc^{k+1}} + \|e^\sigma_{\Bf^*}\|^2_{\Cc^{k+1}} + \|e^\sigma_{\Hf}\|^2_{\Cc^{k+1}}.
\end{align}
\end{subequations}
Since $u$ and $u_\Ec$ are in $\Ac^0$, then $e^u_\Bf = 0$. Similarly, $\sigma$ and $\sigma_\Ec$ are at the end of the complex, so $\|e^\sigma_{\Bf^*}\|^2_{\Cc^{k+1}}$ and $\|e^\sigma_{\Hf}\|^2_{\Cc^{k+1}}$ are both zero.

The error in the harmonic part of $u$ is zero. In the \v{C}ech-de Rham Hodge-Laplacian, we seek a solution which satisfies $u \perp \Hf^0(\Cc)$, hence $u_\Hf = 0$. $u_\Sc$ also satisfies $u_\Sc \perp \Hf^0(\Sc)$ and since $u_{\Ec,0}|_{U_{0,1}} = -u_{\Ec,1}|_{U_{0,1}}$, it follows that $u_\Ec \perp \Hf^0(\Cc)$, and therefore we have that $\|e^u_{\Hf}\| = 0$. 

The only non-zero parts of the errors are then $e^u = e^u_{\Bf^*}$ and $e^\sigma = e^\sigma_\Bf$. The error estimates that are obtained in \cref{theorem: a posteriori} tell us
\begin{subequations}
\begin{align}
\|e^u\|^2 &= \|e^u_{\Bf^*}\|^2 \leq \left(c_\Ac^2\|D^*Du_{\Ec} - f \|\right)^2, \label{eq: a posteriori u} \\ 
\|e^\sigma\|^2  &= \|e^\sigma_{\Bf}\|^2 \leq \left(c_\Ac \|D^*\sigma_\Ec - f\| \right)^2 = \left(c_\Ac \|D^*Du_{\Ec} - f\| \right)^2. \label{eq: a posteriori sigma}
\end{align}
\end{subequations}

In order to get reasonable values for the right-hand side of these inequalities, we consider a smoothed solution $\hat\sigma$. If we were to use the embedded solution $\sigma_\Ec$, then on the overlap, the residuals read
\begin{align}    
D^*\sigma_\Ec - f = d^* \sigma_{\Ec,i} +\delta^* \sigma_{\Ec,(0,1)} - 0.
\end{align}
However, the embedded stress is constant (zero) on the overlap, so the first term vanishes. Because we have inner products weighted by $\epsilon^{-1}$ in $\Ac^{1,0}$, the operator $\delta^*$ absorbs that weight and we are left with a term which blows up as $\epsilon \to 0$.

To avoid this blow-up, we define a smoother stress variable which is non-constant on the overlap (and stays the same as $\sigma_\Ec$ on the non-overlapping region), such that the term $d^* \sigma_i$ and $\delta^*\sigma_{0,1}$ can balance each other. Since $\sigma_{0,1}$ is a constant value, we write $\hat\lambda = \frac{w_{0,1}}{w\epsilon}\hat{\sigma}_{0,1}$, and define $\hat\sigma_i$ on the overlap by:

\begin{subequations}    
\begin{align}
\hat{\sigma}_0(x) &= \sigma_{\Ec,0}(-\epsilon) - \lambda(x+\epsilon), \qquad x \in (-\epsilon, \epsilon), \\
\hat{\sigma}_1(x) &= \sigma_{\Ec,1}(\epsilon) + \lambda(x-\epsilon), \qquad x \in (-\epsilon, \epsilon).
\end{align}
\end{subequations}

When we apply the codifferential to $\hat\sigma_i$, we get
\begin{equation}
d^*\hat\sigma_0 = w\hat\lambda =\frac{w_{0,1}}{\epsilon}\hat{\sigma}_{0,1} = (\delta^*\sigma_{0,1})_0.
\end{equation}
For $\hat\sigma = (\hat\sigma_0, \hat\sigma_1, \sigma_{\Ec, (0,1)})$, we get $D^*\hat\sigma = 0$ exactly on the overlap.

Since $\hat\sigma = \sigma_\Ec$ on $\tilde{U}_i$, we have that $D^*\hat\sigma = d^*\hat\sigma = \pm (1-\epsilon) r$. Together, the residual is then evaluated to
\begin{equation}
\|D^*\sigma_{\Ec} - f \| = \|\mathbbm{1}_{\tilde{U}_i} \left( (1-\epsilon) \pm r \right) -\mathbbm{1} (\pm r) \| = \|\mathbbm{1}_{\tilde{U}_i} (\pm \epsilon r)\| = 2\epsilon r \sqrt{1-\epsilon}. 
\end{equation}
The residual is $O(\epsilon)$, and using $C_{\Ac} \leq \frac{2}{\sqrt{w} \pi}$, we get a reliable and robust estimate for the error of the displacement.

The same smoothed solution $\hat\sigma$ does not give a reliable upper bound for the error $\|e^\sigma\|$, because the error is $O(\sqrt{\epsilon})$, while the residual of $\hat\sigma$ is $O(\epsilon)$. Instead of considering a smoothed solution which balances $d^*$ and $\delta^*$ exactly on the overlap, we seek a smoothed solution $\tilde{\sigma} = (\tilde{\sigma}_0, \tilde{\sigma}_1, \sigma_{\Ec, (0,1)})$ which whose residual is constant on the overlap $(-\epsilon, \epsilon)$, 

Similarly to $\hat\sigma$, we let $\tilde\lambda = \frac{1}{w}(\frac{w_{0,1}}{\epsilon}\hat{\sigma}_{0,1} +1)$, and define

\begin{subequations}    
\begin{align}
\tilde{\sigma}_0(x) &= \sigma_{\Ec,0}(-\epsilon) - \tilde\lambda(x+\epsilon), \qquad x\in (-\epsilon, \epsilon), \\
\tilde{\sigma}_1(x) &= \sigma_{\Ec,1}(\epsilon) + \tilde\lambda(x-\epsilon), \qquad x\in (-\epsilon, \epsilon). 
\end{align}
\end{subequations}

Then $D^*\tilde\sigma - f  = \pm 1$ exactly on the overlap, and equal to $\pm\epsilon r$ outside the overlap. The constant value $1$ on the overlap ensures robustness for the error estimates, since the $O(\epsilon^{1/2})$ residual matches the asymptotics of the error $e^\sigma$.

\Cref{tab: error data} shows the error of the displacement $u$, the efficiency index of \cref{eq: a posteriori u}, the error of the stress $\sigma$ and the efficiency index of \cref{eq: a posteriori sigma}.

\begin{table}[ht]
\centering
\caption{Table of errors, with convergence rates and efficiency indices.}
\label{tab: error data}
\small
\begin{tabular}{c p{1.4cm} p{1.4cm} c p{1.4cm} p{1.4cm} c}
\toprule
$\epsilon$ & Error $e^u$ & Rate $e^u$ & Efficiency $u$ & Error $e^\sigma$ & Rate $e^\sigma$ & Efficiency $\sigma$ \\
\midrule
%4.000e-01 & 9.402e-02 & 2.856 & 6.593e-01 & 2.924 \\
2.000e-01 & 7.311e-02 & 0.766 & 2.032 & 6.130e-01 & 0.430 & 2.057 \\
%1.000e-01 & 4.385e-02 & 1.751 & 4.819e-01 & 1.683 \\
5.000e-02 & 2.384e-02 & 0.941 & 1.632 & 3.562e-01 & 0.487 & 1.479 \\
%2.500e-02 & 1.241e-02 & 1.576 & 2.566e-01 & 1.354 \\
1.250e-02 & 6.330e-03 & 0.985 & 1.549 & 1.828e-01 & 0.498 & 1.274 \\
%6.250e-03 & 3.196e-03 & 1.536 & 1.296e-01 & 1.221 \\
3.125e-03 & 1.606e-03 & 0.996 & 1.529 & 9.165e-02 & 0.500 & 1.185 \\
%1.563e-03 & 8.049e-04 & 1.526 & 6.477e-02 & 1.160 \\
7.813e-04 & 4.029e-04 & 0.999 & 1.525 & 4.577e-02 & 0.500 &  1.143 \\
%3.906e-04 & 2.016e-04 & 1.524 & 3.235e-02 & 1.131 \\
1.953e-04 & 1.008e-04 & $\cdots$ & 1.523 & 2.287e-02 & $\cdots$ & 1.122 \\
\bottomrule
\end{tabular}
\end{table}

\subsection*{Acknowledgments}
DFH and JMN acknowledge the support of the VISTA program, The Norwegian Academy of Science and Letters, and Equinor. JMN acknowledges the support of the Centre for Advanced Study in Oslo, Norway that funded and hosted the research project Mathematical Challenges in Brain Mechanics during the academic year of 2025/2026.

\bibliographystyle{plain}
\bibliography{References.bib}

\appendix
\section{Explicit solution to  equations \eqref{eq: 1d system}}
\label{appendix}

We make use of two assumptions which simplify the equations considerably: we assume that $w_0 = w_1 = w$ and $w_{0,1}$ are constant functions (i.e. the same uniformly distributed material property for the two rods, with a possibly different material property on the overlap) and we assume that $f=(f_0, f_1) = (\mathbbm{1}_{\Tilde{U}_0}r,-\mathbbm{1}_{\Tilde{U}_1}r)$, with $r \in \R$.

We can use the expression $\sigma_{0,1} = u_1|_{U_{0,1}} - u_0|_{U_{0,1}}$ to eliminate one equation and variable. We are left with the following system:
\begin{subequations}\label{eq: 1d system ODE}
\begin{align} 
\frac{du_0}{dx} &= \sigma_0, \\
\frac{du_1}{dx} &= \sigma_1, \\
-w\frac{d\sigma_{0}}{dx} + \frac{w_{0,1}}{\epsilon} \mathbbm{1}_{U_{0,1}} (u_1 - u_0)  &= \mathbbm{1}_{\Tilde{U}_0} r, \label{eq: 1d system ODE d} \\
-w\frac{d\sigma_{1}}{dx} - \frac{w_{0,1}}{\epsilon} \mathbbm{1}_{U_{0,1}} (u_1 - u_0)  &= -\mathbbm{1}_{\Tilde{U}_1}r  \label{eq: 1d system ODE e},
\end{align}
\end{subequations}

The characteristic function $\mathbbm{1}$ is equal to zero outside of $U_{0,1}$ and equal to $1$ on $U_{0,1}$. We can therefore separate the problem into two cases: outside the overlap and on the overlap.

When we are outside $U_{0,1}$, the right-hand side of \cref{eq: 1d system ODE d} and \cref{eq: 1d system ODE e} are equal to zero. The system of equations reduce to
\begin{subequations}
\begin{align}
\frac{du_0}{dx} &= \sigma_0, \\
\frac{du_1}{dx} &= \sigma_1, \\
-w\frac{d\sigma_{0}}{dx}  &= r, \\
-w\frac{d\sigma_{1}}{dx}  &= -r.
\end{align}
\end{subequations}
There is no coupling between the equations, and the solution we obtain is as follows:
\begin{subequations}
\begin{align}\label{eq: solution outside u01}
u_0(x) &= -\frac{r}{2w}x^2 + A_0x + B_0, \qquad &x \in \Tilde{U}_0, \\
u_1(x) &= \frac{r}{2w}x^2 +A_1x+B_1, \qquad &x \in \Tilde{U}_1,\\ 
\sigma_0(x) &= -\frac{r}{w}x + A_0, \qquad &x \in \Tilde{U}_0, \\
\sigma_1(x) &= \frac{r}{w}x+A_1, \qquad &x \in \Tilde{U}_1,
\end{align}
\end{subequations}
with $A_i, B_i \in \R$. The constants $A_i$ are determined by the boundary conditions \cref{eq: 1d boundary conditions}:
\begin{subequations}
\begin{align}
\sigma_0(-1) = \frac{r}{w} + A_0 = 0 \implies A_0 =-\frac{r}{w}, \\
\sigma_1(1) = \frac{r}{w} + A_1 = 0 \implies A_1 =-\frac{r}{w}.
\end{align}
\end{subequations}

On the overlap, we have the system
\begin{subequations}
\begin{align} 
\frac{du_0}{dx} &= \sigma_0, \\
\frac{du_1}{dx} &= \sigma_1, \\
-w\frac{d\sigma_{0}}{dx} + \frac{w_{0,1}}{\epsilon}(u_1 - u_0)  &= 0, \label{eq: 1d OVsystem 3} \\
-w\frac{d\sigma_{1}}{dx} - \frac{w_{0,1}}{\epsilon} (u_1 - u_0)  &= 0 \label{eq: 1d OVsystem 4}.
\end{align}
\end{subequations}
Write $u_\Sigma = u_0 + u_1$ and $\sigma_{0,1} = u_1-u_0$. Then by adding \cref{eq: 1d OVsystem 3} and \cref{eq: 1d OVsystem 4}, and taking the difference \cref{eq: 1d OVsystem 4} minus \cref{eq: 1d OVsystem 3}, we get

\begin{subequations}
\begin{align}
-w \frac{d^2}{dx^2}u_\Sigma &= 0, \\
-w \frac{d^2}{dx^2}\sigma_{0,1}+ \frac{2w_{0,1}}{\epsilon}\sigma_{0,1} &= 0.
\end{align}
\end{subequations}
Solving for $u_\Sigma$ and $\sigma_{0,1}$, we get
\begin{subequations}
\begin{align}
u_\Sigma(x) &= c_1 x + c_2, \\
\sigma_{0,1}(x)&= c_3 e^{\mu x} + c_4e^{-\mu x}.
\end{align}
\end{subequations}
where $\mu = \sqrt{\frac{2w_{0,1}}{w\epsilon}}$.  Since $u_\Sigma = u_0 +u_1$ is an odd function, $c_2 = 0$.

We can use $u_\Sigma$ and $\sigma_{0,1}$ to recover solutions for $u_i$, since $u_0 = \frac12 (u_\Sigma - \sigma_{0,1})$ and $u_1 = \frac12 (u_\Sigma + \sigma_{0,1})$:
\begin{subequations}
\begin{align}
u_0(x)=\frac12(u_\Sigma-\sigma_{0,1}) &= \frac12\left(c_1 x   - c_3 e^{\mu x} - c_4 e^{-\mu x} \right), \\   
u_1(x)=\frac12(u_\Sigma+\sigma_{0,1}) &= \frac12 \left(c_1 x + c_3 e^{\mu x} + c_4 e^{-\mu x}\right). 
\end{align}    
\end{subequations}
Recover the stress $\sigma_i$ on the overlap:
\begin{subequations}
\begin{align}
\sigma_0(x) &= \frac12\left(c_1 - \mu c_3  e^{\mu x} + \mu c_4e^{-\mu x} \right), \\  
\sigma_1(x) &= \frac12 \left(c_1 + \mu c_3 e^{\mu x} -\mu c_4e^{-\mu x}\right). 
\end{align}    
\end{subequations}

Inner boundary conditions and continuity between the solutions outside and inside the overlap determine the remaining constants $c_1, c_3, c_4$:
\begin{subequations}
\begin{align}
\sigma_0(-\epsilon) &=  \frac12\left(c_1 - \mu c_3  E^- + \mu c_4E^+ \right) = \frac{r}{w}(\epsilon - 1), \label{eq: ct a} \\
\sigma_0(\epsilon) &= \frac12\left(c_1 - \mu c_3 E^+ + \mu c_4 E^- \right) = 0, \label{eq: ct b} \\
\sigma_1(-\epsilon) &= \frac12 \left(c_1 + \mu c_3 E^- -\mu c_4 E^+\right) =0, \label{eq: ct c} \\
\sigma_1(\epsilon) &=  \frac12 \left(c_1 + \mu c_3 E^+ -\mu c_4 E^-\right) = \frac{r}{w}(\epsilon - 1). \label{eq: ct d}
\end{align}    
\end{subequations}

Add \cref{eq: ct b} and \cref{eq: ct c} together, and we have 
\begin{equation}
c_1 = \frac\mu2\left(c_3 + c_4\right)(E^+-E^-). \label{eq: 4.22}
\end{equation}
\cref{eq: ct a} minus \cref{eq: ct c} and \cref{eq: ct b} minus \cref{eq: ct d}:
\begin{subequations}
\begin{align}
\sigma_0(-\epsilon)-\sigma_1(-\epsilon) &= \mu \left(-c_3 E^- + c_4 E^+  \right) = \frac{r}{w}(\epsilon - 1), \label{eq: ct2 a} \\
\sigma_0(\epsilon)- \sigma_1(\epsilon) &= \mu\left(-c_3 E^+ + c_4 E^-\right) =- \frac{r}{w}(\epsilon - 1).  \label{eq: ct2 b}
\end{align}
\end{subequations}
Add together \cref{eq: ct2 a} and \cref{eq: ct2 b}, we find that
\begin{align}
 \mu(-c_3 + c_4)(E^+ + E^-) =0 \implies c_3 = c_4.
\end{align}
Putting back $c_3 = c_4$ into \cref{eq: ct2 a}
\begin{subequations}
\begin{align}
\mu \left(-c_3 E^- + c_3 E^+  \right) &= \frac{r}{w}(\epsilon - 1), \\
c_3 (E^+ - E^- )  &=  \frac{r}{w\mu}(\epsilon - 1), \\
c_3 = c_4 &=  \frac{r}{2w\mu\sinh(\mu\epsilon)}(\epsilon - 1).
\end{align}
\end{subequations}
Putting this value for $c_3$ and $c_4$ into \cref{eq: 4.22} determines $c_1$:
\begin{equation}
c_1 = \frac{r}{2w}(\epsilon - 1).
\end{equation}

\end{document}